\documentclass[11pt,reqno]{amsart}
\usepackage[letterpaper,margin=1in,footskip=0.25in]{geometry}
\usepackage{mathrsfs}
\usepackage{amssymb}
\usepackage{microtype}
\usepackage{mathtools}
\usepackage{enumitem}
\usepackage{comment}
\PassOptionsToPackage{dvipsnames}{xcolor}
\usepackage{tikz-cd}
\PassOptionsToPackage{pdfusetitle,pagebackref,colorlinks}{hyperref}
\usepackage{bookmark}
\hypersetup{citecolor=OliveGreen,linkcolor=Mahogany,urlcolor=Plum}
\usepackage{bbm}

\newtheorem{theorem}{Theorem}[section]
\newtheorem{lemma}[theorem]{Lemma}
\newtheorem{proposition}[theorem]{Proposition}
\newtheorem{corollary}[theorem]{Corollary}
\newtheorem{conjecture}[theorem]{Conjecture}

\newtheorem{innercustomthm}{Theorem}
\newenvironment{customthm}[1]
  {\renewcommand\theinnercustomthm{#1}\innercustomthm}
  {\endinnercustomthm}

\theoremstyle{definition}
\newtheorem{definition}[theorem]{Definition}
\newtheorem{example}[theorem]{Example}

\theoremstyle{remark}
\newtheorem{remark}[theorem]{Remark}

\newtheoremstyle{cited}{}{}{\itshape}{}{\bfseries}{\bfseries .}{ }{\thmname{#1} \thmnumber{#2} \thmnote{\normalfont#3}}
\theoremstyle{cited}

\newcommand{\C}{\mathbb{C}}

\newcommand{\Q}{\mathbb{Q}}
\newcommand{\R}{\mathbb{R}}
\newcommand{\Z}{\mathbb{Z}}
\newcommand{\N}{\mathbb{N}}
\renewcommand{\P}{\mathbf{P}}

\newcommand{\fa}{\mathfrak{a}}
\newcommand{\fb}{\mathfrak{b}}
\newcommand{\fm}{\mathfrak{m}}

\newcommand{\cF}{\mathcal{F}}

\newcommand{\cJ}{\mathcal{J}}

\newcommand{\cO}{\mathcal{O}}

\newcommand{\cR}{\mathcal{R}}

\newcommand{\cW}{\mathcal{W}}

\renewcommand{\a}{\alpha}

\renewcommand{\d}{\delta}
\newcommand{\e}{\varepsilon}

\newcommand{\la}{\lambda}

\newcommand{\gt}{\overline{t}}
\newcommand{\gy}{\overline{y}}
\newcommand{\gs}{\overline{s}}
\newcommand{\go}{\overline{o}}
\newcommand{\Xgt}{X_{\overline{t}}}
\newcommand{\Xgy}{X_{\overline{y}}}

\DeclareMathOperator{\codim}{codim}
\DeclareMathOperator{\Div}{Div}
\DeclareMathOperator{\Exc}{Exc}
\DeclareMathOperator{\gr}{gr}
\DeclareMathOperator{\Grass}{Gr}
\DeclareMathOperator{\Hom}{Hom}
\DeclareMathOperator{\lct}{lct}

\DeclareMathOperator{\vol}{vol}

\DeclareMathOperator{\G}{Gr}
\DeclareMathOperator{\Fl}{Fl}
\DeclareMathOperator{\im}{im}
\DeclareMathOperator{\HF}{HF}

\DeclareMathOperator{\ord}{ord}

\DeclareMathOperator{\Proj}{Proj}
\DeclareMathOperator{\redu}{red}
\DeclareMathOperator{\rank}{rank}

\DeclareMathOperator{\DivVal}{DivVal}
\DeclareMathOperator{\Val}{Val}
\DeclareMathOperator{\univ}{u}

\DeclareMathOperator{\Jac}{Jac}

\newcommand{\ab}{\mathfrak{a}_\bullet}
\newcommand{\bs}[1]{\fb\big( | #1 | \big)}

\newcommand{\D}{\Delta}

\newcommand{\Ric}{\mathrm{Ric}}
\newcommand{\bN}{\mathbb{N}}

\newcommand{\bQ}{\mathbb{Q}}
\newcommand{\bP}{\mathbb{P}}
\newcommand{\hvol}{\widehat{\mathrm{vol}}}

\numberwithin{equation}{section}       

\title{Openness  of uniform K-stability in families of $\Q$-Fano varieties
}
\date{\today}

\author{Harold Blum}

\address{Department of Mathematics\\
  University of Utah\\
Salt Lake City, UT 84112, USA.}
\email{blum@math.utah.edu}

\author{Yuchen Liu}

\address{Department of Mathematics\\ 
Yale University\\ 
New Haven, CT 06511, USA}
\email{yuchen.liu@yale.edu}

\begin{document}

 \pagenumbering{arabic}

\maketitle

\begin{abstract}
We show that uniform K-stability is a Zariski open condition in $\Q$-Gorenstein families of $\Q$-Fano varieties. To prove this result, we consider the behavior of the stability threshold in families. The stability threshold (also known as the delta-invariant) is a recently introduced invariant that is known to detect the  K-semistability and uniform K-stability of a $\Q$-Fano variety. We show that the stability threshold is lower semicontinuous in families and provide an interpretation of the invariant in terms of the K-stability of log pairs. 
\end{abstract}

\begin{center}
    \emph{Throughout, we work over a characteristic zero algebraically closed field $k$.}
\end{center}

\section{Introduction}

In this article, we consider the behavior of K-stability in families of $\Q$-Fano varieties.
Recall that K-stability is an algebraic notion introduced by Tian \cite{Tian97} and later reformulated by Donaldson \cite{Don02} to detect the existence of certain canonical metrics on complex projective varieties. In the special case of complex $\Q$-Fano varieties, the Yau-Tian-Donaldson conjecture states that a complex $\Q$-Fano variety is K-polystable iff it admits a K\"ahler-Einstein metric. (By a $\Q$-Fano variety, we mean a projective variety that has at worst klt singularities and  anti-ample canonical divisor.)
 For smooth complex Fano varieties, this conjecture was recently settled in the work of Chen-Donaldson-Sun and Tian \cite{CDS15,Tian15} (see also \cite{DS16,CSW15,BBJ18}).

Another motivation for understanding the K-stability of $\Q$-Fano varieties is to construct  compact moduli spaces for such varieties. It is expected that there is a proper good moduli space parametrizing K-polystable $\Q$-Fano varieties of fixed dimension and volume.  For smoothable $\Q$-Fano varieties, such a moduli space is known to exist \cite{LWX16} (see also \cite{SSY16,Oda15}).  A key step in constructing the moduli space of K-polystable Fano varieties is verifying the Zariski openness of K-semistability. Towards this goal, we prove

\begin{customthm}{A}\label{t:A}
If $\pi:X\to T$ is a projective flat family of varieties such that T is normal, $\pi$ has normal connected fibers, and $-K_{X/T}$ is $\Q$-Cartier and $\pi$-ample, then
\begin{itemize}    \item[(1)] 
$\{
t\in T \, \vert\, X_{\gt}  \text{ is uniformly K-stable} \}$ is a Zariski  open subset of $T$, and
\item[(2)]
$\{
t\in T \,\vert\, X_{\gt}  \text{ is  K-semistable} \}$ is a countable intersection of Zariski open subsets of $T$.
\end{itemize}
\end{customthm}

The notion of uniform K-stability is a strengthening of K-stability introduced in \cite{BHJ1,Der16}.  In \cite{BBJ18}, it was shown that a smooth Fano variety $X$ with discrete automorphism group is uniformly K-stable iff there exists a K\"ahler-Einstein metric on $X$. 
The latter equivalence was later extended 
to  $\Q$-Fano varieties with discrete automorphism group in \cite{LTW19}.
K-semistability is strictly weaker than K-(poly)stability and corresponds to being almost K\"ahler-Einstein \cite{Li17a,BBJ18}.

In \cite{BX18}, the first author and Xu show that the moduli functor of uniformly K-stable $\Q$-Fano varieties of  fixed volume and dimension is represented by a separated Deligne-Mumford stack, which has a coarse moduli space that is a separated algebraic space. The proof of the result combines Theorem \ref{t:A}.1 with a  boundedness statement in \cite{Jia17} (that uses ideas from \cite{Bir16})  and a separatedness  statement in \cite{BX18}.

For smooth families of Fano varieties, Theorem \ref{t:A} is not new. Indeed, for a smooth family of complex Fano varieties with discrete automorphism group,  the K-stable locus is Zariski open by \cite{Oda13,Don15}. In \cite{LWX16}, it was shown that the K-semistable locus  is Zariski open in families of smoothable $\Q$-Fano varieties. These results all rely on deep analytic tools developed in \cite{CDS15,Tian15}. 

Unlike the previous results, our proof of Theorem \ref{t:A} is purely algebraic. (A different algebraic proof of Theorem \ref{t:A}.2  was also given in  \cite{BL18} using a characterization of K-semistability in terms of the normalized volume of the affine cone over a $\Q$-Fano variety \cite{Li17,LL16,LX16}.)  Furthermore, the result holds for all $\Q$-Fano varieties, including those that are not smooth(able), and also log Fano pairs. The argument relies on new tools for characterizing the uniform K-stability and K-semistability of Fano varieties \cite{BHJ1,Li17,Fujitavalcrit,FO16,BJ17}.
\\

More precisely, our approach to proving Theorem \ref{t:A} is through understanding the behavior of the \emph{stability threshold} (also known as \emph{$\delta$-invariant} or \emph{basis log canonical threshold}) in families. We recall the definition of this new invariant.

Let $X$ be projective klt variety and $L$ an ample Cartier divisor on $X$. Set
$$|L|_{\Q} := \{ D \in \Div(X)_{\Q} \, \vert \, D\geq 0\textrm{ and } m D \sim m L \text{ for some } m \in \Z_{>0} \}  .$$
Following \cite{FO16}, we say that $D \in |L|_{\Q}$ is an $m$-basis type divisor of $L$ if there exists a basis $\{s_1,\ldots,s_{N_m} \}$ of  $H^0(X,\cO_X(mL))$ such that 
\[
D = \frac{1}{m N_m} \big( 
\{s_1=0\} + \cdots + \{s_{N_m}=0\} 
\big)
.\]
For $m\in M(L):= \{ m\, \vert\, h^0(X, \cO_X(mL) )\neq 0\}$, set
\[
\d_m(X;L) : = \inf_{D \text{ $m$-basis type}} \lct(X;D) 
,\]
where $\lct(X;D)$ denotes the log canonical threshold of $D$. The stability threshold of $L$ is 
\[
\d(X;L) := \limsup\limits_{M(L) \ni m \to \infty}\d_m(X;L).\] 
In fact, the above limsup is a limit by \cite{BJ17}.
If $X$ is a $\Q$-Fano variety, we set $\delta(X):= r \delta(X;-rK_X)$, where $r\in \Z_{>0}$ is such that $-rK_X$ is Cartier. (The definition is independent of the choice of $r$.) 

The stability threshold is closely related to \emph{global log canonical threshold of} $L$, which is an algebraic version of Tian's $\alpha$-invariant. Recall that the global log canonical threshold of $L$ is 
\[
\alpha(X;L) := \inf_{D\in |L|_{\Q}}\lct(X;D)\]
The two thresholds satisfy
\[
\frac{n+1}{n}\a(X;L)\leq \d(X;L) \leq (n+1) \alpha(X;L)
.\]
where $n=\dim(X)$.

The stability threshold was introduced in the $\Q$-Fano case by K.~Fujita and Y.~Odaka to characterize the K-stability of $\Q$-Fano varieties \cite{FO16}. More generally, the invariant coincides with an invariant suggested by R.~Berman and defined in \cite{BoJ18}. 
Using the valuative criterion for K-stability of K. Fujita and C. Li \cite{Fujitavalcrit,Li17}, it was shown  that the invariant characterizes certain K-stability notions.

\begin{theorem}\label{t:FOBJ}\cite{FO16,BJ17}
Let $X$ be a $\Q$-Fano variety. 
\begin{itemize}
    \item[(1)] $X$ is uniformly K-stable iff $\d(X)>1$.  
    \item[(2)] $X$ is K-semistable iff $\d(X)\geq1$. 
\end{itemize}
\end{theorem}

In light of the previous statement, 
Theorem \ref{t:A} is a consequence of  the following result.

\begin{customthm}{B}\label{t:B}
Let $\pi:X\to T$ be a projective flat family of varieties and $L$ a $\pi$-ample Cartier divisor on $X$. Assume $T$ is normal, $X_t$ is a klt variety for all $t\in T$, and $K_{X/T}$ is $\Q$-Cartier. Then, the two functions $T\to \mathbb{R}$ defined by
\[
T \ni t  \mapsto  \d(X_{\gt};L_{\gt})
\quad \text{ and }
\quad
T \ni t  \mapsto  \a(X_{\gt};L_{\gt})
\]
are lower semicontinuous.
\end{customthm}

In the above theorem,  $(X_{\gt};L_{\gt})$  denotes the restriction of $(X,L)$ to the geometric fiber over $t$. 
As explained in \cite[Remark 4.15]{CP18}, the above result would not hold with ``$\delta(X_{\gt};L_{\gt})$'' replaced by ``$\delta(X_{t};L_t)$.''

Let us note the main limitation of Theorem \ref{t:B}. While the statement implies
$
\{ t \in T \, \vert \, \d(X_{\gt};L_{\gt}) >
a
\}$
is open for each $a\in \R_{\geq 0}$, it does not imply $t \mapsto \d(X_{\gt};L_{\gt})$ takes finitely  many values. Hence, we are unable to prove 
$
\{ t \in T \, \vert \, \d(X_{\gt};L_{\gt}) \geq
a
\}$
is open and cannot verify the openness of
K-semistability  in families of $\Q$-Fano varieties. 
The openness of K-semistability is an immediate  consequence of Theorem \ref{t:B} and the following conjecture
(see \cite[Conjecture 2]{BL18} for a local analogue).

\begin{conjecture}
If $\pi:X\to T$ is a projective family of varieties such that $T$ is normal, $X_t$ is klt for all $t\in T$, and $-K_{X/T}$ is $\Q$-Cartier and ample, then
$T \ni t \mapsto \d(X_{\gt})$
takes finitely many values. 
\end{conjecture}

We also provide a new interpretation of the stability threshold in terms of (log) K-stability. The result provides further motivation for studying this invariant. 
Note that a similar result is obtained independently
by Cheltsov, Rubinstein and Zhang in \cite[Lemma 5.8]{CRZ18}.

\begin{customthm}{C}\label{t:C}
Let $X$ be a $\Q$-Fano variety. 
We have:
\begin{align*}
\min\{1, \d(X)\} &= \sup \{ \beta \in (0,1] 
 \, \vert \, (X,(1-\beta )D) \text{ is K-semistable for some } D\in |-K_X|_{\Q} \}\\
            & = \sup \{ \beta \in (0,1] 
 \, \vert \, (X,(1-\beta) D) \text{ is uniformly K-stable for some } D\in |-K_X|_{\Q} \}
\end{align*}

\end{customthm}

To conclude the introduction, we briefly explain the proof of Theorem \ref{t:B} for the stability threshold. The strategy is similar in spirit to the proof of \cite[Theorem 1]{BL18}.
\begin{itemize}
    \item[(1)] We define a modification of $\d_m(X_{\gt},L_{\gt})$, denoted by $\widehat{\d}_m(X_{\gt},L_{\gt})$, defined in terms of $\N$-filtrations of  $H^0(X_{\gt},\cO_X(mL_{\gt}))$ rather than bases of this vector space (see \S \ref{ss:stabfilt} for the precise definition). The advantage of working with $\N$-filtrations of $H^0(X_{\gt},\cO_X(mL_{\gt}))$ is that $\N$-filtrations of bounded length are simply flags. Hence, they are parametrized by a proper variety.
     \item[(2)] We show  $\widehat{\d}_m$ is lower semicontinuous for $m\gg0$ (Proposition \ref{p:dmlsc}) and $(\widehat{\d}_m)_m$ converges to $\d$ as $m \to \infty $ (Theorem \ref{p:dmdmhat}).
 
    \item[(3)] To show that $\d$ is lower semicontinuous, it is sufficient to show that $(\widehat{\d}_m)_m$ converges to $\d$ uniformly.  
    We prove a slightly weaker convergence result (Theorem \ref{t:deltaconv}) which also implies
the lower semicontinuity of $\delta$. The statement is an extension of a convergence result in \cite{BJ17} whose proof relies on Nadel vanishing and properties of multiplier ideals.  
    \end{itemize}

\noindent \emph{Postcript note}: Since the first version of this paper was posted on the arXiv, there have been  developments on the above topics.
In \cite{BLX19}, the authors and Xu proved that, in the setting of Conjecture 1.2, the function $T\ni t \mapsto \min \{1, \delta(X_t) \}$ is constructible. This result
combined with Theorem B implies the openness of K-semistability in families of $\mathbb{Q}$-Fano varieties.\footnote{
In a separate paper, Xu gave an independent proof of the openenss of K-semistability 
by first proving that the normalized volume of a klt singularity is constructible in families \cite{Xu19}.} \\

    \noindent \emph{Acknowledgment}.
    We thank 
    Mattias Jonsson, J\'anos Koll\'ar, Chi Li, Mircea Musta\c{t}\v{a}, and  Chenyang Xu for many useful discussions.
    We also thank Giulio Codogni, Yanir Rubinstein, Song Sun, Gang Tian, and the anonymous referee for many helpful comments. The research of HB was partially supported by NSF grant DMS-1803102.

\section{Preliminaries}

\subsection{Conventions} 
We work over an algebraically closed characteristic zero field $k$.
A \emph{variety} will mean 
an integral separated scheme of finite type over $k$. For a variety $X$, a point $x\in X$ will mean a 
scheme theoretic point. A geometric point $\overline{x} \in X$ will mean a map from the spectrum of an algebraically closed field to $X$.

A \emph{pair} $(X,\Delta)$ is a composed of a normal variety $X$ and an effective $\Q$-divisor $\D$ such that $K_{X}+\D$ is $\Q$-Cartier. 
If $(X,\D)$ is a pair and $f:Y\to X$ a proper birational morphism with $Y$ normal, we write $\Delta_Y$ for the $\Q$-divisor on $X$ such that 
\[
K_Y+\D_Y = f^*(K_X+\D).\]  

Let $(X,\D)$ be a pair and $f:Y\to X$ a log resolution of $(X,\D)$. 
The pair $(X,\D)$ is \emph{lc}  (resp., \emph{$\e$-lc}) if $\Delta_Y$ has coefficients $\leq 1$ (resp., $\leq 1-\e$). The pair $(X,\D)$ is \emph{klt} if $\D_Y$ has coefficients $<1$. Hence, klt implies $\varepsilon$-lc for some $\varepsilon >0$.

A pair $(X,\D)$ is log Fano if $X$ is projective, $-(K_{X}+\D)$ is ample, and $(X,\D)$ is klt. A variety $X$ is \emph{$\Q$-Fano} if $(X,0)$ is a log Fano pair.

\subsection{K-stability}
Let $(X,\D)$ be a projective pair such that $-K_X-\D$ is ample. We refer the reader to \cite{BHJ1} for the definition of K-semistability  and uniform K-stability of $(X,\D)$ in terms of test configurations.\footnote{While these notions are defined for polarized pairs, we will always mean K-stability with respect to the anti log canonical polarization $L=-K_X-\D$.} In this article, we will use a characterization of  K-semistability  and uniform K-stability in terms of the stability threshold (see Theorem \ref{t:dk-s}). 

\subsection{Families of klt pairs}
A \emph{$\Q$-Gorenstein family of klt pairs} $\pi:(X,\D) \to T$ over a normal base will mean a  flat surjective morphism of varieties $\pi: X\to T$ and a $\Q$-divisor $\D$ on $X$ not containing any fibers satisfying: 
\begin{itemize}
    \item[(1)] $T$ is normal and $f$ has normal, connected fibers (hence, $X$ is normal as well),
    \item[(2)] $K_{X/T}+\D$ is $\Q$-Cartier, and
    \item[(3)]  $(X_t,\D_t)$ is a klt pair for all $t\in T$.
\end{itemize}
 
We briefly explain the definition of  $\D_t$ mentioned above. 
Let $U\subseteq X$ denote the smooth locus of $f$.
The assumption that $K_{X/T}+\D$ is $\Q$-Cartier implies $\D\vert_{U}$ is $\Q$-Cartier on $U$, while the assumption that $X_t$ is normal implies $\codim( X_t , X_t \setminus (X_t \cap U)) \geq 2$. Hence, we may define $\D_t$ as the unique $\Q$-divisor on $X_t$ whose restriction to $X_t \cap U$ is the pullback of $\Delta_U$ to $X_{t} \cap U$.

%
%
%

\subsection{Valuations}
Let $X$ be a variety. A \emph{valuation on $X$} will mean 
a valuation $v\colon K(X)^\times \to \R$ that is trivial on $k$ and has center on $X$. Recall, $v$ has \emph{center}   on $X$ if there exists a point $\xi \in X$ such that $v\ge0$ on $\cO_{X,\xi}$
and $v>0$ on the maximal ideal of $\cO_{X,\xi}$. Since $X$ is assumed to be separated, such a point $\xi$ is unique, and we say $v$ has center $c_{X}(v):= \xi$.
We use the convention that $v(0) = +\infty$. 

We write $\Val_{X}$ for the set valuations on $X$, and $\Val_X^*$ for the set of non-trivial valuations. The set $\Val_X$ may be equipped with the topology of pointwise convergence as in \cite{jonmus,BdFFU}, but we will not use this additional structure.

To any valuation $v\in \Val_X$ and $\la\in \R$ there is an associated \emph{valuation ideal}
 $\fa_\la(v)$ defined as follows.  For an affine open subset $U\subseteq X$,  
$\fa_{\la}(v)(U)=\{ f \in \cO_X(U)  \, \vert \, v(f) \ge\la \} $ if $c_{X}(v) \in U$ and $\fa_\la(v)(U) = \cO_{X}(U)$ otherwise.

For an ideal $\fa\subseteq \cO_X$ and $v\in \Val_X$,  we set 
\[
v(\fa) := \min\{v(f)\, \vert\, f\in \fa \cdot \cO_{X, c_X(v)} \} \in [0, +\infty].\]
We can also make sense of $v(s)$ when $L$ is a line bundle and $s\in H^0(X, L)$. After trivializing $L$ at $c_X(v)$, we write $v(s)$ for the value of the local function corresponding to $s$ under this trivialization;
this is independent of the choice of trivialization.

Similarly, when $D$ is a Cartier divisor, we set $v(D):=v(f)$ where $f$ is a local equation for $D$ at $c_X(v)$. When $D$ is a $\Q$-Cartier $\Q$-divisor, we set $v(D):=m^{-1}v(mD)$, where $m\ge 1$ is chosen so that $mD$ is a Cartier divisor.

%
%
%
%

\subsection{Divisorial valuations}
If $\pi:Y\to X$ is a proper birational morphism, with $Y$ normal, and $E\subset Y$
is a prime divisor (called a \emph{prime divisor over $X$}), then $E$ defines a
valuation $\ord_E\colon K(X)^ \times \to\Z$ in $\Val_X$ given by the order of vanishing at the generic point of $E$. 
Note that $c_X(\ord_E)$ is the generic point of $\pi(E)$. 
Any valuation of the form 
$v=c \cdot \ord_E$ with $c\in\R_{>0}$ will be called \emph{divisorial}. We write $\DivVal_X \subset \Val_X$ for the set of divisorial valuations.

%
%
%

%
%
%
%
\subsection{Graded sequences of ideals}
A \emph{graded sequence of ideals} is a sequence $\fa_\bullet=(\fa_p)_{p \in \Z_{>0}}$ 
of ideals on $X$ satisfying $\fa_p\cdot \fa_q\subseteq \fa_{p+q}$ for all $p,q\in \Z_{>0}$. We will always assume $\fa_p\ne (0)$ for some $p\in \Z_{>0}$.  
We write $M(\ab) := \{p \in \Z_{>0}\mid\fa_p \neq (0)\}$.  
By convention,  $\fa_{0}:=\cO_X$. 

Let $\fa_\bullet$ be a  graded sequence of ideals on $X$
and $v\in{\Val}_X$. It follows from Fekete's Lemma that the limit 
\[
  v(\fa_\bullet) := \lim_{M(\ab)\ni m\to\infty} \frac{v(\fa_m)}{m}
\]
exists, and equals $\inf_{m} v(\fa_m) /m$; see \cite{jonmus}. 

The following statement concerns a type of graded sequence of ideals that will arise in \S\ref{subs:extending}. 

\begin{proposition}\label{prop:gord}
Let $\fa_1,\ldots,\fa_{p}$ be nonzero ideals on a variety $X$. For each $p \in \Z_{>0}$, 
set
\[
\fb_{p}:= \sum_{b } \fa_{1}^{b_1} \cdot \cdots \cdot \fa_{m}^{b_{m}},\]
where the sum runs through all $b=(b_1,\ldots, b_{m})\in \N^{m}$ such that 
$\sum_{i=1}^{m} i b_i = p$. The following hold:
\begin{itemize}
    \item[(1)] $\fb_\bullet$ is a graded sequence of ideals on $X$. 
    \item[(2)] There exists $N\in\Z_{>0}$ such that $\fb_{Np} = \fb_{N}^p$  for all $p \in \N$.
\end{itemize}
\end{proposition}

\begin{proof}
Statement (1) is clear. To approach (2), 
consider the polynomial ring 
$$R: = \bigoplus_{p\in \N} R_p=  k[X_1,\ldots,X_m]$$
with grading
 given by setting $\deg(X_i)=i$ (i.e. $R_p$ is spanned by monomials of 
the form 
{$ X_1^{b_1}\cdots X_{m}^{b_m}$ where $ \sum_{i=1}^{m} i b_i = p$}).
Since $R$ is finitely generated over $R_0$, there exists a positive integer $N$
so that $R_{Np}=R_{N}^p$ for all $p>0$.

With the above choice of $N$, 
we claim that $\fb_{Np}= \fb_{N}^p$ for all $p>0$. 
Clearly, $ \fb_{N}^p \subseteq \fb_{Np}$. For the reverse inclusion, fix $p>0$ and choose $b=(b_1,\ldots,b_{m})\in \N^{m}$ such that $\sum i b_i = N p$. To finish the proof, we will  show $ \fa^b:= \fa_1^{b_1}\cdots \fa_{m}^{b_{m}} \subseteq \fb_{N}^p$. 

Using that $R_{Np}= R_{N}^p$ and $X^b:=X_1^{b_1} \cdots X_m^{b_m} \in R_{Np}$, we may find $c^{(1)}, \ldots,c^{(p)}\in \N^{m}$ such that 
$X^{c^{(j)}} := {X_1}^{c^{(j)}_1}\cdots {X_m}^{c^{(j)}_m} \in R_{N}$ for each $1\leq j\leq p$
and 
$X^b=    X^{c^{(1)}}  \cdots  X^{c^{(p)}}$.
This translates to say  \[
\fa^{c^{(j)}} := \fa_1^{c_1^{(j)}}\cdots \fa_{m}^{c_{m}^{(j)}}  \subseteq \fb_N \quad\text{ for each $1\leq j \leq p$} \]
and 
$\fa^b = \fa^{c^{(1)}} \cdots \fa^{c^{(p)}}$. Therefore, 
$
\fa^b = \fa^{c^{(1)}} \cdots \fa^{c^{(p)}} \subseteq \fb_N^p$
and the proof is complete.
\end{proof}

%
\subsection{Log discrepancies}
Let $(X,\Delta)$ be a pair. If $\pi \colon Y\to X$ is a projective birational morphism with $Y$ normal and $E\subset Y$ 
a prime divisor, then the \emph{log discrepancy} of $\ord_E$  with respect to $(X,\D)$ is defined by
$$A_{X,\D}(\ord_E):= 1- (\text{coefficient of $E$ in } \D_Y).$$ Following \cite{jonmus,BdFFU}, the function $A_{X,\D}: \DivVal_X \to \R$ may be extended to a lower semicontinuous function $A_{X,\D}: \Val_X \to \R \cup \{ + \infty\}$ (see \cite[\S3.2]{Blu18} for the setting of log pairs). When the choice of the pair $(X,\Delta)$ is clear from context, we will sometimes write $A(v)$ for  $A_{X,\Delta}(v)$ to reduce notation.

We will frequently use the following facts: A pair $(X,\D)$ is klt iff $A_{X,\D}(v)>0$ for all  $v\in \Val_{X}^*$. 
If $(X,\D)$ is a pair and $D$ an effective $\Q$-Cartier $\Q$-divisor on $X$, then  $A_{X,\D+D}(v)= A_{X,\D}(v) -v(D)$ \cite[Proposition 3.2.4]{Blu18}.

%

\subsection{Log canonical thresholds}
Let $(X,\D)$ be a klt variety. Given a nonzero ideal $\fa \subseteq \cO_X$, the \emph{log canonical threshold} of $\fa$ is given by 
\begin{equation*}
\lct(X,\D;\fa):= \inf_{v \in \DivVal_X} \frac{A_{X,\D}(v)}{v(\fa)} = \inf_{v \in \Val_{X}^* } \frac{A_{X,\D}(v)}{v(\fa)} .
\end{equation*}
If $f:Y \to X$ is a log resolution of $(X,\D,\fa)$, then the above infimum is achieved by a divisorial valuation $\ord_E$, where $E$ is a divisor on $Y$. 
If $D$ is a $\Q$-divisor on $X$, then 
\begin{equation*}
\lct(X,\D;D):= \inf_{v \in \DivVal_X} \frac{A_{X,\D}(v)}{v(D)} = \inf_{v \in \Val_{X}^* } \frac{A_{X,\D}(v)}{v(D)} 
\end{equation*}
and is equal to $\sup \{ c \in \R_{>0} \, \vert \, (X,\D+cD) \text{ is lc}\}$.

Let $\ab$ be graded sequence of ideals on $X$. Following \cite[\S 3.4]{Blu18} (which extends result of \cite{jonmus} to the setting of klt pairs), 
the \emph{log canonical threshold} of $\ab$ is given by
\begin{equation*}
  \lct(X,\D,\ab):= \lim_{M(\ab) \ni m \to \infty}  m \cdot \lct(X,\D;\fa_m) = \sup_{m \ge 1} m \cdot \lct(X,\D;\fa_m).
\end{equation*}
 We have 
\begin{equation*}
  \lct(X,\D;\ab)= \inf_{v\in\DivVal_X} \frac{A_{X,\D} (v)}{v(\ab)}
  =
  \inf_{v\in\Val_X^*} \frac{A_{X,\D} (v)}{v(\ab)}
\end{equation*}
by \cite[Propositions 3.4.3-3.4.4]{Blu18}. 
We say $v^\ast \in \Val_X$ \emph{computes} $\lct(X,\D;\ab)$ if it computes the previous infimum. Given a graded sequence $\ab$, such a valuation always exists  \cite[Theorem A]{jonmus} \cite[Theorem 3.4.10]{Blu18}.

\begin{lemma}\cite[Lemma 3.4.9]{Blu18}\label{l:abvlct}
If $v\in \Val_X^*$, then $\lct(X,\D;\ab(v))\leq A_{X,\D}(v) $.
\end{lemma}

\section{Filtrations}

In this section, we recall information on filtrations of section rings. Much of the content appears in \cite[\S 2]{BJ17}.

Throughout, let $X$ be a normal projective variety of dimension $n$ and $L$ a big Cartier divisor on $X$. Write 
\[
R=R(X,L)= \bigoplus_{m\in \N} R_m =\bigoplus_{m \in \N} H^0(X, \cO_{X}(mL))\]
for the section ring of $L$. 
Set 
$$N_m := \dim H^0(X,\cO_X(mL)) \quad \text{  and }\quad M(L):= \{m \in \N \, \vert \, H^0(X,\cO_{X}(mL) ) \neq 0 \}.$$

\subsection{Graded linear series}

A \emph{graded linear series} 
$W_\bullet =\{W_m \}_{m \in \Z_{>0}}$ of $L$ is a collection of $k$-vector subspaces 
$W_m \subseteq H^0(X,\cO_{X}(mL) ) $
such that
\[
R(W_\bullet):= \bigoplus_{m \in \N} W_m  \subseteq \bigoplus_{m \in \N} R_m 
 \]
 is a graded sub-algebra of $R(X,L)$. By convention, $W_0:= H^0(X,\cO_X)$.

A graded linear series $W_\bullet$ of $L$ is said to be \emph{birational} if for all $m \gg0$, 
 $W_m \neq 0$ and the rational map $X \dashrightarrow \P(W_m)$
 is birational onto its image.
  A graded linear series $W_\bullet$ of $L$ is said to \emph{contain an ample series} if: $W_m\neq 0 $ for all $m \gg 0$, and there exists a decomposition  $L = A + E$ where $A,E$ are $\Q$-divisors with $A$ ample and $E$ effective such that 
 \[
 H^0(X, \cO_X(mA) ) \subseteq W_m \subseteq H^0(X,\cO_X(mL))\]
 for all $m$ sufficiently large and divisible.
 If $W_\bullet$ contains an ample series, then $W_\bullet$ is birational.

\begin{example}\label{ex:lingls}
 Fix a vector subspace $V\subseteq H^0(X,\cO_{X}(L))$. 
\begin{itemize} 
\item[(1)] For each $m > 0$, set $V_{m} := \im (S^{m} V \to H^0(X,mL) )$.  Then $V_{\bullet}$ is a graded linear series of $L$ and $R(V_\bullet)$ is a finitely generated $k$-algebra. If the rational map  $X \dashrightarrow \P(V)$ is birational, then the graded linear series $V_\bullet$ is birational.
\item[(2)]
For each $m>0$, set $\widetilde{V}_{m} = H^0(X,mL \otimes  \fb_m)$, where $\fb_m$ denotes the integral closure of the $m$-th power of the base ideal\footnote{The base ideal of $|V|$ is the ideal $\fb( |V|) :={\rm im}( \cO_X(-L) \otimes_k V\to \cO_{X})\subseteq \cO_X$.} of $|V|$. 
Now, $\widetilde{V}_\bullet$ is a graded linear series of $L$. If  the rational map  $X \dashrightarrow \P(V)$ is birational, then the graded linear series $\widetilde{V}_{\bullet}$ is birational.
\end{itemize}
\end{example}

\subsection{Volume of graded linear series}

 Let $W_\bullet$ be a graded linear series of $L$. 
The  \emph{Hilbert function}  of $W_\bullet$ is the function $HF_{W_\bullet} :\N \to \N$ defined by 
\[HF_{W_\bullet}(m) = \dim(W_m).\] When $V\subseteq H^0(X,\cO_X(L))$ is a linear series, we set $HF_{V}:= HF_{V_\bullet}$, where $V_\bullet$ is the graded linear series defined in Example \ref{ex:lingls}.1. 
 
 The \emph{volume} of $W_\bullet$ is given by
 \[
 \vol(W_\bullet):= 
 \limsup\limits_{ M(W_\bullet) \ni  m \to \infty} \frac{ \dim W_m}{m^n/n!}
 =
 \limsup\limits_{M(W_\bullet) \ni m \to \infty} \frac{ HF_{W_\bullet} (m)}{(m^n/ n!) }
,\]
where $M(W_\bullet ) := \{m \in \N \,\vert\, \dim(W_m) \neq 0 \}$. The previous limsups are in fact limits \cite{LM09,KK12}.

\begin{proposition}\label{p:volbir}
Let $V\subseteq H^0(X,\cO_X(L))$ be a nonzero vector subspace and $\pi:Y \to X$ a proper birational morphism with $Y$ normal such that $\bs{V} \cdot \cO_Y= \cO_Y(-E)$ with $E$ a Cartier divisor on $Y$.
If the map $X \dashrightarrow \P(V)$ is birational, then 
$\vol(V_{\bullet}) = \vol( \widetilde{V}_{\bullet})= (\pi^*L-E)^n$.
\end{proposition}

\begin{proof}
We first show $\vol(V_{\bullet})=(\pi^*L-E)^n$. Consider the rational map $\varphi :X\dashrightarrow \P(V)$ and write $Z$ for the closure of the image.  The rational map extends to a morphism $\widetilde{ \varphi}:Y  \to \P(V)$ with the property $\widetilde{ \varphi}^*\cO_{\P(V)}(1)\simeq  \pi^* L-E$. Since $Z = \Proj ( R(V_\bullet))$ and $\widetilde{ \varphi}$ is birational, 
\[
\vol(V_\bullet) =  \cO_Z(1)^n = (\widetilde{ \varphi}^* \cO_Z(1) )^n=  (  \pi^* L-E)^n.\]

We next show $\vol( \widetilde{V}_{\bullet})= (\pi^*L-E)^n$. 
 Since $\pi_* \cO_Y(-mE)\subseteq \cO_X$ is the integral closure of the $m$-power of  $\bs{V}$, $
\widetilde{V}_{m} \simeq H^0(Y, \cO_Y(m(\pi^*L-E)))$. Hence, $\vol(\widetilde{V}_{\bullet})=\vol(\pi^*L-E)$. Since $\pi^*L-E$ is base point free and, hence, nef, $\vol(\pi^* L-E)= (\pi^*L-E)^n$.  
\end{proof}

\subsection{Filtrations}
\begin{definition}\label{def:filt}
For $m \in \N$, a \emph{filtration} $\cF$ of $R_m$ we will mean a family  of $k$-vector subspaces $\cF^\bullet R_m =(\cF^\la R_m)_{\la \in \R_{\geq 0}}$ of $R_m$ such that 
\begin{itemize}
\item[(F1)]
$\cF^\la R_m  \subseteq \cF^{\la'} R_m$ when $\la \geq \la'$;
\item[(F2)]
$\cF^\la R_m = \cap_{\la ' < \la} \cF^{\la'} R_m$ for $\la>0$;
\item[(F3)]
$\cF^0 R_m =R_m$ and $\cF^\la R_m=0$ for $\la \gg0$.
\end{itemize}
 A \emph{filtration} $\cF$ of $R$ is the data of a filtration $\cF$ of
 $R_m$ for each $m\in \N$ such that 
\begin{itemize}
\item[(F4)]
$\cF^\la R_m \cdot \cF^{\la'} R_{m'} \subseteq \cF^{\la +\la'} R_{m+m'}$
for all $m,m'\in \N$ and $\la,\la'\in \R_{\geq 0}$. 
\end{itemize}
A filtration $\cF$ of $R_m$ is trivial if $\cF^\la R_m=0$ for all $\la>0$. A filtration $\cF$ of $R$ is trivial if $\cF^\bullet R_m$ is trivial for all $m \in \N$.
\end{definition}

\subsection{Jumping numbers.}
Let $\cF$ be a filtration of $R_m$ where $m \in M(L)$. The \emph{jumping numbers} of $\cF$
are given by 
\[
0 \leq a_{m,1}\leq \cdots \leq a_{m,N_m} = mT_{m}(\cF)\]
where 
\[
a_{m,j} = a_{m,j}(\cF) = \inf\{ \la \in \R_{\geq 0} \, \vert \, \codim \cF^\la R_m \geq j \} \]
for $1\leq j \leq N_m$. 
The scaled average of the jumping numbers and the maximal jumping number are given by
\[
S_{m}(\cF): = \frac{1}{mN_m} \sum_{j=0}^{N_m} a_{m,j}(\cF) \quad \text{ and } \quad 
T_{m}(\cF) := \frac{a_{m,{N_m}}}{m}.\]

\subsection{Induced graded linear series}
Given a filtration $\cF$ of $R$, there is an induced family of graded linear series
$ V_{\bullet}^{\cF,s}$
indexed by $s\in \R_{\geq0}$ and defined by
\[
V_{m}^{\cF,s} := \cF^{ms}H^0(X,\cO_{X}(mL)).\] 
To reduce notation, we will often write  $V_{\bullet}^s$ for $V_{\bullet}^{\cF,s}$ when the choice of filtration is clear.

By unraveling our definitions, we see \[
T_m(\cF) =  \sup \{ s \in \R_{\geq 0} \, \vert\, V_m^{\cF,s}\neq 0 \},
\quad
\text{ and } 
\quad
S_m(\cF) = \frac{1}{ N_m} \int_0^{T_m(\cF)} \dim V_m^{\cF,s} \, ds
\]
for $m \in M(L)$.
Since property (F4) implies $T_{m_1+m_2}(\cF) \geq \frac{m_1}{m_1+m_2} T_{m_1}(\cF) + \frac{m_2}{m_1+m_2} T_{m_2}(\cF)$, the limit 
\[
T(\cF):= \lim_{M(L) \ni m \to \infty} T_m (\cF) \in [0,+\infty]
\]
exists by Fekete's Lemma \cite[Lemma 2.3]{jonmus} and equals $\displaystyle \sup_{m\in M(L)} T_m(\cF)$. 
We say $\cF$ is \emph{linearly bounded} if $T(\cF) <+ \infty$.


The following two propositions are a consequence of \cite[\S 1.3]{BC11}. For the second proposition, see \cite[Lemma 2.9]{BJ17} for the result stated in our terminology.

\begin{proposition}\label{p:BC}
Let $\cF$ be a linearly bounded filtration of $R$. 
\begin{itemize}
    \item[(1)] $V_\bullet^{\cF,s}$ contains an ample series for $s\in [0,T(\cF))$.
    \item[(2)] The function $s\mapsto \vol(V_\bullet^{\cF,s})^{1/n}$ is a decreasing concave function on $[0,T(\cF)]$ and vanishes on $(T(\cF), +\infty)$. 
\end{itemize}
\end{proposition}


\begin{proposition}
For any linearly bounded filtration $\cF$ of $R$, we have 
\[
\lim_{M(L)\ni  m\to \infty}S_m(\cF)=
\frac{1}{\vol(L)}\int_{0}^{T(\cF)}\vol(V_\bullet^{\cF,s})\, ds.\]
\end{proposition}

Given the above proposition, we set  $S(\cF):= \lim_{M(L)\ni m \to \infty} S_m(\cF)$. The following lemma follows easily from our definitions.

\begin{lemma}
Let $\cF$ be a linearly bounded filtration of $R$. We have:
\begin{itemize}
    \item[(1)] $0  \leq S_m(\cF)\leq  T_m(\cF)$ for all $m \in M(L)$.
    \item[(2)] $0 \leq S(\cF) \leq T(\cF)$.
\end{itemize}
\end{lemma}

We next consider a variant of $S_m(\cF)$ that is more asymptotic in nature. For $s\in [0, T(\cF))$ and $m \in M(L)$, consider the graded linear series $\widetilde{V}_{m,\bullet}^{\cF,s}$, where 
\[
\widetilde{V}_{m,k}^{\cF,s} = H^0\left(X, \cO_X(kmL) \otimes \overline{\bs{V_m^{\cF,s}}^k}\right)
\]
as in Example \ref{ex:lingls}.2. We set 
\[
\widetilde{S}_m(\cF):=
\frac{1}{\vol(L)}\int_{0}^{T(\cF)} 
\frac{
\vol(\widetilde{V}_{m,\bullet}^{\cF,s})}{m^n} \, ds.
\]
\begin{proposition}
For any linearly bounded filtration $\cF$ of $R(X,L)$, we have 
\[
S(\cF) = \lim_{M(L)\ni m \to \infty} \widetilde{S}_m(\cF).\]
\end{proposition}

\begin{proof}
We claim that for $s\in [0,T(\cF))$, 
\begin{equation}\label{e:VFs1}
\vol(V_{\bullet}^{\cF,s}) = \lim_{m \to \infty } \frac{ \vol(\widetilde{V}_{m,\bullet}^{\cF,s})}{m^n}.
\end{equation}
If we  assume the claim and note that $ \vol(\widetilde{V}_{m,\bullet}^{\cF,s})/ m^n \leq \vol(L)$, we see that the proposition now follows from the dominated convergence theorem.

To prove the above claim, note that $V_{\bullet}^{\cF,s}$ contains an ample series for $s\in [0,T(\cF))$ by Proposition \ref{p:BC}.1. Now, we may apply \cite[Theorem D]{LM09} to see 
\begin{equation}\label{e:VFs2}
\vol(V_{\bullet}^{\cF,s})= \lim_{m\to \infty } \frac{ \vol(V_{m,\bullet}^{\cF,s})}{m^n}
,\end{equation}
where $V_{m,p}^{\cF,s}: = \im( S^p V_m^{\cF,s} \to R_{mp})$ as in Example \ref{ex:lingls}.1. Combining \eqref{e:VFs2} with Proposition \ref{p:volbir} completes the claim. 
\end{proof}



\subsection{Filtrations induced by valuations}
Given $v\in \Val_{X}$, we set 
\[
\cF_v^\la R_m = \{ s\in H^0(X, \cO_{X}(mL)) \, \vert \, v(s) \geq \la \}\]
for each $\la\in \R_{\geq 0 }$ and $m \in \N$. 
Equivalently, $\cF_v^\la R_m = H^0(X,\cO_{X}(mL) \otimes \fa_{\la}(v))$. Note that $\cF_v$ is a filtration of $R$. 

\begin{proposition}\cite[Lemma 5.2.1]{Blu18}
Let $(X,\Delta)$ be a projective klt pair and $L$ a big Cartier divisor on $X$. If $v\in \Val_{X}$ and $A_{X,\D}(v)< +\infty$, then the filtration $\cF_v$ of $R(X,L)$ is linearly bounded. 
\end{proposition}

\begin{definition}
 Let $v$ be a valuation on $X$ such that  $\cF_v$ is a linearly bounded filtration of $R$.
 \begin{itemize}
     \item[(1)] The \emph{maximal vanishing} (or \emph{pseudo-effective threshold}) of $L$ along $v$ is $T(L;v):= T(\cF_v)$.
     \item[(2)] The \emph{expected vanishing} of $L$ along $v$ is $S(L;v):= S(\cF_v)$. 
 \end{itemize}
\end{definition}

When the choice of $L$ is clear, we  simply write $T(v)$ and $S(v)$ for the $T(L;v)$ and $S(L;v)$. Similarly, we also write $T_m(v)$, $S_m(v)$, and $\widetilde{S}_m(v)$ for  $T_m(\cF_v)$, $S_m(\cF_v)$, and $\widetilde{S}_m(\cF_v)$.

\begin{remark}
 Let $\pi:Y \to X$ be a proper birational morphism with $Y$ normal. If $E$ is prime divisor on $Y$, then 
 \[
 S(L;\ord_E) : = \frac{1}{\vol(L)} \int_0^\infty \vol(\pi^*L-xE)\, dx\]
 and
 \[
 T(L;\ord_E) : = \sup \{x \in \R_{>0} \, \vert \, \pi^*L-xE \text{ is pseudo-effective} \}. \]
These invariants played an important role in the work of C. Li \cite{Li17} and  K. Fujita \cite{Fujitavalcrit}. In the notation of \cite{Fujitavalcrit}, 
\[
T(L; \ord_E) = \tau(L; \ord_E) \quad \text{ and } \quad S(L; \ord_E) = \tau(E) - \vol(L)^{-1} j(E)
\]
when $X$ is $\Q$-Fano and $L=-K_X$.
\end{remark}

\begin{proposition}\label{p:STprops} \cite[Lemma 3.7]{BJ17}
Let $v$ be a valuation on $X$ of linear growth. 
\begin{itemize}
    \item[(1)] For $c\in \R_{>0}$, $S(L;cv)= c S(L;v)$ and  $T(L;cv)= c T(L;v)$.
    \item[(2)]For $m\in \Z_{>0}$, $S(mL;v)= m S(L;v)$ and  $T(mL;v)= T(L;v)$.
    \item[(3)] If $\pi:Y \to X$ is a projective birational morphism with $Y$ normal, then $S(\pi^*L;v) = S(L;v)$ and $T(\pi^*L;v) =T(L;v)$. 
\end{itemize}

\end{proposition}

\begin{remark}
If $L$ is a big $\Q$-Cartier $\Q$-divisor on $X$ and $v\in \Val_{X}$ is a valuation of linear growth, then we set 
$S(L; v) := (1/m) S(mL;v)$, where $m\in \Z_{>0}$ is chosen so that $mL$ is a Cartier divisor. By Proposition \ref{p:STprops}.2, $S(L;v)$ is independent of the choice of $m$.
\end{remark}

\subsection{$\N$-filtrations}
\begin{definition}
A filtration $\cF$ of $R_m$  is an $\N$-filtration if all its jumping numbers are integers. Equivalently, 
\[
\cF^\la R_m = \cF^{\lceil \la \rceil}R_m 
\]
for all $\la \in \R_{\geq 0 }$. 

An \emph{$\N$-filtration} of $R$ is a filtration of $\cF$ of $R$ if $\cF^\bullet R_m$  is an $\N$-filtration for each $m \in \N.$
Note that an $\N$-filtraton of $R$ is equivalent to 
the data of subspaces $(\cF^\lambda R_m)_{m,\lambda \in \N}$ such that (F1), (F3), and (F4) of Definition \ref{def:filt} are satisfied.
\end{definition}

Any filtration $\cF$ of $R$ induces an $\N$-filtration $\cF_{\N}$ defined by setting
\[
\cF_{\N}^{\la}R_m := \cF^{\lceil \la \rceil}R_m .\]
Indeed, conditions (F1)-(F3) are trivially satisfied for $\cF_\N$ and (F4) follows from the inequality $\lceil \la_1 \rceil + \lceil \la_2 \rceil \geq \lceil \la_1 + \la_2 \rceil$.  

\begin{proposition}\cite[Proposition 2.11]{BJ17}\label{p:Nfilt}
If $\cF$ is a filtration of $R$ with linear growth, then 
\[
T_m(\cF_\N) = \lfloor m · T_m(\cF)\rfloor /m \hspace{.5 cm} \text{ and }
\hspace{.5 cm} S_m(\cF) -1/m \leq S_m(\cF_\N) \leq  S_m(\cF) .
\]
Hence, $S(\cF) = S(\cF_\N)$ and  $T(\cF) = T(\cF_\N)$.
\end{proposition}

\subsection{Base ideals of filtrations.} In this subsection, we assume $L$ is ample. 
To a filtration $\cF$ of $R$, we associate a graded sequence of base ideals. For $\la\in \R_{\geq 0}$ and $m\in M(L)$, set
\[
\fb_{\la,m}(\cF) := \fb( \vert \cF^\la H^0(X,\cO_{X}(mL)) \vert). \]

\begin{lemma}
\cite[Lemma 3.17 and Corollary 3.18]{BJ17}
The sequence of ideals $(\fb_{\la,m}(\cF))_{m\in M(L)}$ has a unique maximal element, which we denote by $\fb_{\la}(\cF)$. Furthermore, 
\begin{itemize}
    \item [(1)] $\fb_{\la}(\cF) = \fb_{\la,m}(\cF)$ for $m\gg0$, and
    \item[(2)] $\fb_{\bullet} (\cF)= (\fb_{p} (\cF))_{p \in \N}$ is a graded sequence of ideals.
\end{itemize}
\end{lemma}

We state some basic properties of these ideal sequences. 

\begin{lemma}\label{lem:baseFv}\cite[Lemma 3.19]{BJ17}
If $v\in \Val_{X}$, then $\fb_\la (\cF_v)= \fa_{\la}(v)$ for all $\la \in \R_{\geq 0 }$. 
\end{lemma}

\begin{proposition}\label{prop:SvFcomp}
Let $v$ be a valuation on $X$ and $\cF$ a filtration of $R$. If $\cF_v$ and $\cF$ are both of linear growth, we have 
\[
S(v)\geq v(\fb_\bullet(\cF)) S(\cF) \hspace{0.5 cm}
\text{ and}
\hspace{0.5 cm}
T(v)\geq v(\fb_\bullet(\cF)) T(\cF).\]
In the case when $\cF$ is an $\N$-filtration, 
    \[
S_m(v)\geq v(\fb_\bullet(\cF)) S_m(\cF) \hspace{0.5 cm}
\text{ and}
\hspace{0.5 cm}
T_m(v)\geq v(\fb_\bullet(\cF)) T_m(\cF)
\]
for all $m \in M(L)$.
\end{proposition}

\begin{proof}
It is sufficient to prove the inequalities after replacing $v$ with a  scalar multiple. Hence, we may consider the case when $v(\fb_\bullet(\cF))=1$. Now, \cite[Lemma 3.20]{BJ17} gives 
\begin{equation}\label{eq:ffv}
\cF^{p}R_m \subseteq \cF_v^p R_m 
\end{equation}
for all $m\in M(L)$ and $p\in \N$. Therefore, 
\[
a_{p,m}(\cF_{\N})\leq  a_{p,m}(\cF_{v,\N}) 
\]
for all $m\in M(L)$ and $0 \leq p\leq N_m$. The previous inequality combined with Proposition \ref{p:Nfilt} gives 
\[
S_m(\cF_\N) \leq S_m(\cF_{v,\N}) \leq S_m(\cF_v):= S_m(v) . 
\]
If $\cF= \cF_\N$ (which is the case when $\cF$ is an $\N$-filtration), 
we see $S_m(\cF) \leq S_m(\cF_v)$. More generally, Proposition \ref{p:Nfilt} implies $S(\cF) \leq S(v)$. The inequalities for $T_m(\cF)$ and $T(\cF)$ follow from the same argument. \end{proof}

\subsection{Extending filtrations.}\label{subs:extending}
In this subsection, we again assume $L$ is ample. 
Fix $m' \in M(L)$ and consider a $\N$-filtration $\cF$ of $R_{m'}$. Set $r':= m' T_{m'} (\cF)$.

\begin{definition}
We write $\widehat{\cF}$ for the $\N$-filtration of $R$ defined as follows:
\begin{itemize}
    \item[(i)] For $m<m'$, 
    \[
    \widehat{\cF}^p R_m := 
    \begin{cases}
    R_m & \text{ for } p=0\\
    0 & \text{ for } p>0 
    \end{cases}.\hspace{1.5 cm}
    \]
    \item[(ii)] For $m= m'$, 
\[    \widehat{\cF}^p R_m := \cF^p R_m \text{ for } p\geq 0.
\]
    \item[(iii)] For $m>m'$, 
     \begin{equation}\label{eq:extfilt}
 \widehat{\cF}^{p} R_{m}
:= \sum_{  b   } \left(  (\cF^{1} R_{m'})^{b_1} \cdot \cdots (\cF^{r'} R_{m'})^{b_{r'}}\right) \cdot R_{m - m' \sum b_i },
\end{equation}
where the previous sum runs through all $b=(b_1,\ldots,b_{r'}) \in \N^{r'}$ such that 
$\sum_{i=1}^{r'} i b_{i} = p$ and $ m \geq   m' \sum_{i=1}^{r'} b_i$.
\end{itemize}
It is clear that $\widehat{\cF}$ is a filtration of $R$. Furthermore, $\widehat{\cF}$  is the minimal filtration of $R$ such that $\widehat{\cF}$ and $\cF$ give the same filtration of $R_{m'}$. 
\end{definition}

\begin{remark}
 The previous definition is similar to the definition of $\chi^{(k)}$ in \cite[\S 3.2]{Sze15}, though the the conventions for filtrations in \emph{loc. cit.}
 differ slightly from those in this paper.
\end{remark}

\begin{lemma}\label{l:extfilt}
 The following hold:
\begin{itemize}
    \item[(1)] $S_{m'}(\cF) = S_{m'}(\widehat{\cF})$. 
    \item[(2)]
   Let $\fa_i$ denote the base ideal of $ \cF^i R_{m'}$ for each $1\leq i \leq r'$.
    For each $p>0$, 
\[
\fb_p( \widehat{\cF}  ) = \sum_{b} \fa_1^{b_1}\cdot \cdots \cdot \fa_{r'}^{b_{r'}} 
\]
where the sum runs through all $b=(b_1,\ldots,b_{r'}) \in \N^{r'}$ such that $
\sum_{i=1}^{r'} i b_i = p $.
\end{itemize}
\end{lemma}

\begin{proof}
Statement (1) follows immediately from the fact that $\cF$ and $\widehat{\cF}$ give the same filtration on $R_{m'}$.
We now show (2). Taking base ideals of the left and right sides of \eqref{eq:extfilt} gives the inclusion ``$\subseteq$''. 
For the reverse inclusion, fix  $b =(b_1,\ldots, b_{r'})\in \N^{r'}$ such that $\sum_{i=1}^{r'} i\cdot  b_i = p$. We will proceed to show
$\fa_1^{b_1}\cdot \cdots\cdot \fa_{r'}^{b_{r'}}\subseteq \fb_p( \hat{\cF})$.

Choose $M \in \N$ so that $\cO_{X}(mL)$ is globally generated for all $m \geq M$. Now, if $
 m  -  m'\sum_{i=1}^{r'} b_i\geq M$, then  \eqref{eq:extfilt} gives
\[
\fa_1^{b_1}\cdot \cdots \cdot \fa_{r'}^{b_{r'}}  \subseteq \fb_{p,m}(\widehat{\cF})
.\]
Since $\fb_p (\widehat{\cF}) = \fb_{p,m}(\widehat{\cF})$ for $m\gg0$, we conclude $\fa_1^{b_1}\cdot \cdots \cdot \fa_{r'}^{b_{r'}}  \subseteq\fb_p (\widehat{\cF})$. 
\end{proof}


\section{Thresholds}

Let $(X,\D)$ be a klt pair and $L$ a big Cartier divisor on $X$. Associated to $L$ are two thresholds that measure the singularities of members of $|mL|$ as $m \to \infty$.

\subsection{The global log canonical threshold}

For $m \in M(L)$, we set 
\[
\alpha_m(X,\Delta;L) = \inf_{D\in |mL|}m\lct( X,\D, D) .\]
The \emph{global log canonical threshold} of $L$ is 
\[
\alpha(X,\D;L) = \inf_{m \in M(L)}\alpha_{m}(X,\D,L).\]
When the choice of pair $(X,\D)$ is clear, we will often write $\alpha(L)$ for the above threshold.
As explained in \cite[Theorem A.3]{CS08}, the global log canonical threshold can be interpreted analytically as a generalization of the $\alpha$-invariant introduced by Tian.

The global log canonical threshold may  be expressed in terms of valuations \cite{Amb16,BJ17}. (See \cite{Blu18} for the level of generality stated below.)

\begin{proposition}\label{p:alphamvals}
For $m \in M(L)$,
\[
\alpha_m(X,\D,L) = \inf_{v\in \DivVal_{X} } \frac{ A_{X,\D}(v)}{T_m(v)}
=
 \inf_{v } \frac{ A_{X,\D}(v)}{T_m(v)}
,\]
where the second infimum runs through all valuations $v\in \Val_{X}^*$ with $A_{X,\D}(v)< +\infty$. 
\end{proposition}

\begin{proposition}\label{p:alphavals}
We have
\[
\alpha(X,\D,L) = \inf_{v\in \DivVal_{X} } \frac{ A_{X,\D}(v)}{T(v)}
=
 \inf_{v } \frac{ A_{X,\D}(v)}{T(v)}
,\]
where the second infimum runs through all valuations $v\in \Val_{X}^*$ with $A_{X,\D}(v)< +\infty$. 
\end{proposition}

\subsection{The stability threshold} 
Given $m \in M(L)$, we say that $D\in |L|_{\Q}$ 
is an $m$-basis type divisor of $L$ if there exists a basis $\{s_1,\ldots, s_{N_m} \}$ of $ H^0(X,\cO_{X}(mL))$ such that 
\[
D = \frac{1}{mN_m} \left( 
\{s_1=0\} + \cdots + \{ s_{N_m}=0\} \right).\]
Set 
\[
\delta_m(X,\D;L) := \inf \{ \lct(X,\D;D) \, \vert\, D \text{ is a $m$-basis type divisor of $L$} \}.\]
The \emph{stability threshold} of $L$ is 
\[
\delta(X,\D;L) := \limsup \limits_{M(L)\ni m \to \infty} \delta_m(X,\D;L) 
.\]
When the pair $(X,\D)$ is clear, we will simply write $\delta(L)$ for $\delta(X,\D;L)$. 

The previous definition of stability threshold was introduced in \cite{FO16} by K. Fujita and Y. Odaka in the log Fano case. The invariant was designed to characterize the K-stability of log Fano varieties in terms of singularities of anti-canonical divisors. 

\begin{proposition}\label{p:dmvals}\cite[Proposition 4.3]{BJ17}
For $m\in M(L)$,
\[
\d_m(X,\D;L) = \inf_{v\in \DivVal_{X} } \frac{ A_{X,\D}(v)}{S_m(v)}
=
 \inf_{v } \frac{ A_{X,\D}(v)}{S_m(v)}
,\]
where the second infimum runs through all valuations $v\in \Val_{X}^*$ with $A_{X,\D}(v)< +\infty$.
\end{proposition}

\begin{theorem}\cite[Theorem C]{BJ17}\label{prop:deltavals}
We have
\[
\d(X,\D;L) = \inf_{v\in \DivVal_{X} } \frac{ A_{X,\D}(v)}{S(v)}
=
 \inf_{v } \frac{ A_{X,\D}(v)}{S(v)}
,\]
where the second infimum runs through all valuations $v\in \Val_{X}^*$ with $A_{X,\D}(v)< +\infty$.  Furthermore, the limit $\lim_{M(L)\ni m\to \infty} \delta_m(X,\D;L)$ exists and equals $\delta(X,\Delta;L)$.
\end{theorem}

\begin{remark}\label{r:min}
If we further assume that the  base field $k=\C$  and $L$ is ample, there exists $v^*\in \Val_{X}^*$ with $A_{X,\D}(v^*)< +\infty$
such that $\d(X,\D;L) = A_{X,\D}(v^*)/S(v^*)$ \cite[Theorem E]{BJ17}. We will not use this result.
\end{remark}

\begin{remark}
We can also make sense of $\d(X,\D;L)$ when $L$ is a big $\Q$-Cartier $\Q$-divisor. In this case, we set 
\[
\delta(X,\D;L) : = r \delta(X,\D;rL) ,
\]
where $r\in \Z_{>0}$ is chosen so that $rL$ is a Cartier divisor. As a consequence of Theorem \ref{prop:deltavals} and Proposition \ref{p:STprops}.2, $\delta(X,\D;L)$ is independent of the choice of $r$.
\end{remark}

\begin{proposition}\cite[Theorem A]{BJ17}\label{p:alpdelin}
We have 
\[
\alpha(X,\D;L) \leq \d(X,\D;L) \leq (n+1) \a(X,\D;L)
,\]
where $n=dim(X)$. Furthermore, when $L$ is ample $((n+1)/n)\alpha(X,\D;L)  \leq \d(X,\D;L)$. 
\end{proposition}

When $(X,\D)$ is a log Fano pair, we set 
\[
\d(X,\D):=\d(X,\D;-K_X-\D)
.\]
Using K. Fujita and C. Li's valuative criterion for (log) K-stability \cite{Fujitavalcrit,Li17}, Theorem \ref{prop:deltavals} implies

\begin{theorem}\cite[Theorem 0.3]{FO16} \cite[Theorem B]{BJ17}\label{t:dk-s}
Let $(X,\D)$ be a log Fano pair. 
\begin{itemize}
    \item[(1)] $(X,\D)$ is K-semistable iff $\delta(X,\D)\geq 1$.
    \item[(2)] $(X,\D)$ is uniformly K-stable iff $\delta(X,\D)>1$.
\end{itemize}
\end{theorem}

\begin{remark}
In \cite{BJ17}, the previous statements were proved in the case when $\D=0$. The more general case follows from the same argument (see \cite{Blu18,CP18}). 

Note that in \cite{BJ17} the result is proven when $k=\C$. While the uncountability of the base field is needed to prove \cite[Theorem E]{BJ17} (see Remark \ref{r:min}), the above result holds over any algebraically closed characteristic zero field.
\end{remark}

\subsection{The stability threshold in terms of filtrations}\label{ss:stabfilt}

We now proceed to interpret the stability threshold in terms of filtrations. We restrict ourselves to the case when $L$ is ample.

\begin{proposition}\label{p:dFilts}
If $(X,\D)$ is a projective klt pair and $L$ an ample Cartier divisor on $X$, then
\[
\delta(X,\D;L) = 
\inf_{\cF} \frac{\lct(X,\D; \fb_\bullet(\cF) )}{ S(\cF)}
,\]
where the infimum runs through all non-trivial linearly bounded filtrations of $R$
\end{proposition}

\begin{remark}
In \cite{BoJ18},   $\d(X,\D;L)$ is  expressed in terms of Radon probability measures on the Berkovich analytification of $X$. Such probability measures are closely related to filtrations of $R$. 
\end{remark}

\begin{proof}
The statement is an immediate consequence of the following lemma and Theorem \ref{prop:deltavals}. 
\end{proof}

\begin{lemma}
Let $(X,\D)$ be a projective klt pair and $L$ an ample Cartier divisor on $X$. 
\begin{itemize}
    \item[(1)] If $v\in \Val_{X}^*$ with $A_{X,\D}(v)<+ \infty$, then
    \[
\frac{    \lct(X,\D;\fb_\bullet(\cF_v))}{ S(\cF_v)} \leq \frac{A_{X,\D}(v)}{S(v) }.
    \]
    \item[(2)] If $\cF$ is a non-trivial filtration of $R(X,L)$ with $T(\cF) <+\infty$ and $w\in \Val_{X}$ computes $\lct(\fb_\bullet(\cF))$, then 
    \[
    \frac{A_{X,\D}(w)}{S(w)}\leq \frac{ \lct(X,\D;\fb_\bullet(\cF))}{ S(\cF)}. 
    \]
\end{itemize}

\end{lemma}

\begin{proof}
In order to prove  (1), note that Lemmas \ref{l:abvlct} and \ref{lem:baseFv} combine to show
\[
 \lct(\fb_\bullet(\cF_{v})) = \lct(\ab(v)) \leq A(v). \]
Since $S(v): = S(\cF_{v})$, the desired inequality follows.

For (2), recall that  $w$ computes $\lct(\fb_\bullet(\cF))$ means $\lct(\fb_\bullet(\cF)) = A(w) / w(\fb_\bullet(\cF))$. 
Combining the previous relation with Proposition \ref{prop:SvFcomp} completes the proof. 
\end{proof}

Next, we introduce a variation on $\delta_{m}(X,\D;L)$, which is defined using filtrations of $R_m$ rather than bases of this vector space.

\begin{definition}
  For $m \in M(L)$, set
\[
\widehat{\delta}_m(X,\D;L) = \inf_{ \cF} \frac{ \lct(X,\D; \fb_\bullet(\widehat{\cF}))}{S_m(\cF)},
\]
where the infimum runs through all non-trivial $\N$-filtrations $\cF$ of $R_m $ with $T_m(\cF)\leq 1$. (Recall, $\widehat{\cF}$ is the minimal  extension of $\cF$ to a filtration of $R$ as defined in \S \ref{subs:extending}.)
\end{definition}

\begin{theorem}\label{prop:hatdelta}
If $(X,\D)$ is a projective klt pair and $L$ an ample Cartier divisor on $X$, then 
\[
\d(X,\D;L) = \lim_{ M(L) \ni m \to \infty} \widehat{\delta}_m(X,\D;L)
.\]
In particular, the above limit exists.
\end{theorem}

To prove the above theorem, we will use the following statements.

\begin{lemma}\label{lem:hatdelta}
Keep the assumptions of Theorem \ref{prop:hatdelta}.
Fix $v\in \Val_{X}$ with $A_{X,\D}(v)<+\infty$. For $m \in M(L)$, let $\cF_{v,m}$ denote the $\N$-filtration of $R_m$ given by $\cF_{v,m}^{\la} R_m : = \cF_{v}^{\lceil \la \rceil}R_m$. The following hold:
\begin{itemize}
    \item[(1)] $\lct(X,\D; \fb_\bullet(\widehat{\cF}_{v,m})) \leq A_{X,\D}(v)$ and
    \item[(2)] $ S_m(v) -m^{-1} \leq S_m(\cF_{v,m})$.
\end{itemize}
\end{lemma}

\begin{proof}
To show (1), we first note that $\fb_p(\widehat{\cF}_{v,m}) \subseteq \fa_p(v) $ for all $p \in \N$, since $\bs{\cF_v^p R_m}\subseteq \fa_{p}(v)$. Therefore, $\lct( \fb_\bullet(\widehat{\cF}_{v,m}) \leq \lct(\ab(v))$. Since $\lct(\ab(v)) \leq A_{X,\D}(v)$ by Lemma \ref{l:extfilt}.2, (1) is complete. Statement (2) follows from Proposition \ref{p:Nfilt} and the fact that $S_m(\cF_{v,m})=S_m(\cF_{v,\N})$.
\end{proof}

\begin{lemma}\label{lem:hatdelta2}
Keep the assumptions of Theorem \ref{prop:hatdelta}.
Fix $m \in M(L)$. 
If $\cF$ is a non-trivial $\N$-filtration of $R_m$ and $w$ computes $\lct(X,\D;\fb_\bullet(\widehat{\cF}))$, then 
\[
\frac{
A_{X,\D}(w)}{S_m(w)} \leq \frac{\lct(X,\D;\fb_\bullet(\widehat{\cF}))}{S_m(\cF)}. 
\]
\end{lemma}

\begin{proof}  Since $w$ computes $\lct(\fb_\bullet(\widehat{\cF}))$, $\lct(\fb_\bullet(\widehat{\cF})) = A_{X,\D}(w) / w(\fb_\bullet(\widehat{\cF}))$. 
Combining the previous inequality with Proposition \ref{prop:SvFcomp} completes the proof. 
\end{proof}

\begin{proposition}\label{p:dmdmhat}
Keep the assumptions of Theorem \ref{prop:hatdelta}.
For $m \in M(L)$,
\[
\frac{1}{
\d_m(X,\D;L)}
-\frac{1}{m \cdot \a(X,\D;L)}
\leq 
\frac{1}{
\widehat{\d}_m(X,\D;L)}
\leq \frac{1}{\d_{m}(X,\D;L)}.\]
\end{proposition}

\begin{proof}
We begin by showing the first inequality.
Fix $\e>0$ and
choose $v\in \Val_{X}^*$ such that $A_{X,\D}(v) < +\infty$ and $A_{X,\D}(v)/S_m(v) < \delta_m(X,\D;L)+\e$. After replacing $v$ with a scalar multiple, we may assume $T(v) =1$. 
Let $\cF_{v,m}$ denote the $\N$-filtration of $R_m$ as defined in Lemma
\ref{lem:hatdelta}. Note that the assumption $T(v) =1$ implies $T(\cF_{v,m}) \leq 1$. 
Therefore, 
\begin{align*}
 \widehat{\d}_m(X,\D;L)^{-1} 
&\geq \frac{S_m(\cF_m)}{\lct(X,\D; \fb_\bullet(\widehat{\cF}_{v,m}))}\\
&\geq    
\frac{S_m(v) -1/m}{A_{X,\D}(v)}
\intertext{by Lemma \ref{lem:hatdelta}.
Now, our choice of $v$ implies}
&\geq
\frac{1}{\delta_m(X,\D;L)+ \e} -  \frac{1}{m \cdot  A_{X,\D}(v)}
\intertext{
and the inequality $\alpha(X,\D;L) \leq A_{X,\D}(v)/ T(v)$ combined with $T(v)=1$  gives} 
& \geq  \frac{1}{\delta_m(X,\D;L)+ \e} -  \frac{1}{m\cdot  \alpha(X,\D;L)}.\end{align*}
Sending $\e \to 0$ completes the first inequality.

We move on to the second inequality. Let $\cF$ be a nontrivial $\N$-filtration of $R_m$ satisfying $T_m(\cF) \leq 1$. After choosing $w\in \Val_{X}^*$ computing $\lct(\fb_\bullet(\widehat{\cF}))$, we apply Lemma \ref{lem:hatdelta2} to see 
\[
\frac{
\lct(\fb_\bullet(\widehat{\cF}))}{S_m(\cF)}
\geq 
\frac{A_{X,\D}(w)}{S_m(w)}
\geq 
\delta_m(X,\D;L),
\]
where the last inequality follows from Proposition \ref{p:dmvals}. 
Hence, $\widehat{\d}_m(X,\D;L) \geq \d_m(X,\D,L)$ and the proof is complete. 
\end{proof}

\begin{proof}[Proof of Theorem \ref{prop:hatdelta}]
The statement follows immediately from combining Theorem \ref{prop:deltavals} with Proposition \ref{p:dmdmhat}.
 \end{proof}

\section{Convergence results}

The goal of this section is to prove the following convergence results.

\begin{theorem}\label{t:alphaconv}
Let $\pi:(X,\D)\to T$ be a projective $\Q$-Gorenstein family of klt pairs over a normal base and $L$ a $\pi$-ample Cartier divisor on $X$. 
For each $\e>0$, there exists a positive integer $m_0=m_0(\e)$ such that 
\[
0\leq \alpha_m(\Xgt,\D_{\gt};L_{\gt})-
\alpha(\Xgt,\D_{\gt};L_{\gt})
\leq \e 
\]
for all positive integers $m$ divisible by $m_0$ and $t\in T$.
\end{theorem} 

\begin{theorem}\label{t:deltaconv}
Let $\pi:(X,\D)\to T$ be a projective $\Q$-Gorenstein family of klt pairs over a normal base  and $L$ a $\pi$-ample Cartier divisor on $X$. 
For each $\e>0$, there exists a positive integer $m_0=m_0(\e)$ such that 
\[
\widehat{\delta}_m(\Xgt,\D_{\gt};L_{\gt})
-
\delta(\Xgt,\D_{\gt};L_{\gt})
\leq  \e 
\]
for all positive integers $m$ divisible by $m_0$ and $t\in T$. 
\end{theorem}

To prove the above theorems, we will first show uniform convergence results for $S$ and $T$ in families (see Corollary \ref{c:STfam}.1 and  Theorem \ref{t:SSm}). 
While the two uniform convergence results for $S$ and $T$ will be deduce from a result in \cite{BJ17} that holds on a fixed variety (see Theorem \ref{t:fujita}), proving the result for $S$ will be significantly more involved.

\subsection{Bounding the global log canonical threshold in families}
We will now prove the following boundedness statement for the global log canonical threshold in bounded families. The result is well known to experts (for example, see \cite[Proposition 2.4]{Oda13} for a special case).

\begin{proposition}\label{p:alphabound}
Let $\pi:(X,\D) \to T$ be a projective $\Q$-Gorenstein family of klt pairs over a normal base  and $L$ a $\pi$-ample Cartier divisor on $X$. There exist constants $c_1,c_2>0$ so that 
\[
c_1 < \alpha(\Xgt,\D_{\gt};L_{\gt}) < c_2
\]
for all $t\in T$. 
\end{proposition}

Before proving the result, we  prove the following lemmas.

\begin{lemma}\label{l:alphaf^*L}
Let $(X,\D)$ be a projective klt pair that is $\e$-log canonical and $L$ an ample Cartier divisor on $X$. If $f:Y \to X$ a log resolution of $(X,\D)$, then:
\begin{itemize}
\item[(1)]  $\e \cdot A_{Y,0}(v) \leq  A_{X,\D}(v)$ for all $v\in \Val_{X}$, and
\item[(2)] $  \e \cdot \alpha(Y,0;f^* L) \leq  \alpha(X,\D; L)$. 
\end{itemize}
\end{lemma}

\begin{proof}
For statement (1), we recall an argument in \cite[Proof of Theorem 9.14]{BHJ1}. 
Since $(X,\D)$ is $\e$-lc, $\D_Y \leq (1- \e)\D_{Y,\redu}$.
Since $\D_{Y,\redu}$ is snc, $(Y,\D_{Y,\redu})$ is lc
and, hence $v(\D_{Y,\redu}) \leq A_{Y,0}(v)$ for all $v\in \Val_{X}$. Therefore,
\[
v(\D_Y) \leq (1-\e) v(\D_{Y,\redu}) \leq( 1-\e)A_{Y,0}(v), \]
and we see 
\[
\varepsilon A_{Y,0}(v) \leq A_{Y,0}(v) - v(\Delta_Y) = A_{X,\Delta}(v),\]
which completes (1).
 Statement (2)  follows from combining  (1)
 with Propositions \ref{p:STprops}.3 and \ref{p:alphavals}.
\end{proof}

\begin{lemma}\label{l:alphaVie}
Let $X$ be a smooth projective variety of dimension $n$ and $L$ an ample Cartier divisor on $X$. If $A$ is a very ample Cartier divisor on $X$, then $\alpha(X,0;L) \geq 1 / (L \cdot A^{n-1}+1)$.  
\end{lemma}

\begin{proof}
See \cite[Corollary 5.11]{Vie95}.
\end{proof}

\begin{proof}[Proof of Proposition \ref{p:alphabound}]
We will show  there exists a dense open set $U\subseteq T$ and constants $c_1,c_2>0$ so that $ c_1 < \alpha(X_{\gt},\Delta_{\gt};L_{\gt}) <c_2$ for all $t\in U$. By induction on the dimension of $T$, the proposition will follow.

Let $f:Y\to X$ be a projective log resolution of $(X,\D)$ and write
\[ K_{Y/T} + \D_Y = f^*(K_{X/T}+\D).\]
 Choose a dense affine open set $U\subseteq T$ such that $U$ is smooth, $Y \to X$ is smooth over $U$, and $\Exc(\pi) +\tilde{\D}$ is a relative snc divisor  over $U$ (see \cite[Definition 2.9]{GKKP}). Thus, $Y_t \to X_t$ is a log resolution of $(X_t,\Delta_t)$ for all $t\in U$. Since the fibers of $(X,\D)$ along $\pi$ are klt, we may find $0 < \e \ll 1$ so that $\D_Y \vert_U$ has coefficients $\leq 1-\e$. Hence, $(X_t,\Delta_t)$ is $\e$-lc for all $t\in U$. 

Now, since $Y_U\to U$ is projective and $U$ is affine, there exists a Cartier divisor $A$ on $Y_U$ that is very ample over $U$. Replacing $A$ with a high enough power, we may assume $f^*L+ A$ is very ample over $U$ as well. For each $t\in U$, we have
\[
\alpha(X_t,\Delta_t, L_t) \geq \e \cdot \alpha(Y_t,0,f^*L_t)  \geq  \alpha(Y,0, f^*L_t + A_t ) \geq 1/ \left( (f^*L_t+A_t)\cdot A_t^{n-1} +1\right) ,\]
where the first inequality holds by Lemma
\ref{l:alphaf^*L}, the second by \cite[Lemma 5.3.6]{Blu18}, and the last 
  by Lemma \ref{l:alphaVie}.
Since $Y_U \to U$ is flat, $U \ni t \mapsto (f^*L_t+A_t)\cdot A_t^{n-1}$ is locally constant. Hence,
there exists $c_1>0$ so that $ \alpha(X_{\gt},\Delta_{\gt};L_{\gt}) > c_1$ for all $t\in U$.

We move onto finding an upper bound. Since $L$ is $\pi$-ample and $U$ is affine, there exists a divisor $\Gamma \in |mL_U|$ for some $m \in \Z_{>0}$ such that $\Gamma$ does not contain a fiber. Now, $t\mapsto \lct(X_{\gt},\D_{\gt};\Gamma_{\gt})$ takes finitely many values \cite[Lemma 8.10]{KP17} and 
\[
\alpha(X_{\gt}, \Delta_{\gt};L_{\gt}) \leq 
\lct(X_{\gt},\D_{\gt}; m^{-1}\Gamma_{\gt})
\]
for all $t\in U$. Hence, there exists $c_2>0$ satisfying the desired inequality, and the proof is complete.\end{proof}

\subsection{A finiteness result for Hilbert functions}
\label{ss:finitenesshilbert}
We now prove the following finiteness result on Hilbert functions in families. A consequence of the statement (Corollary \ref{c:volconv}) will be used in Section \ref{ss:refinedapprox}.

\begin{theorem}\label{thm:finiteW}
Let $\pi: X \to T$ be a flat projective family of varieties and $L$ a $\pi$-ample   Cartier divisor on $X$ such that $R^i \pi_* (  \cO_X(mL)) = 0$ for all $i,m\geq 1$. 
Then the set of functions 
\[
\bigcup_{t \in T} \big\{ HF_{W}:\N \to \N \, \vert \, W \subseteq H^0(X_{\gt},\cO_{\Xgt}(L_{\gt}))\big\}\] 
is finite.
\end{theorem}

The theorem is a consequence of the following proposition and the use of the  Grassmannian to parametrize the set of linear series in question.

\begin{proposition}\label{prop:Wchain}
Keep the setup of Theorem \ref{thm:finiteW} and fix a sub $\cO_T$-module $\cW\subseteq \pi_* \cO_{X}(L)$. For each  geometric point $\gt \in T$, set
$$
W_{\gt}: = \im \left( \cW \otimes k(\gt) \to 
 H^0(X_{\gt},\cO_{X_{\gt}}(L_{\gt}) \right).$$
Then the set of functions 
$\{ \HF_{W_{\gt}}: \N \to \N \, \vert \, \gt \in T \}$
is finite.
\end{proposition}

\begin{proof}
We will prove that there exists a dense open set $U \subseteq T$ 
such that $\HF_{W_{\gt}}$  is independent of geometric point $\gt\in U$.
By induction on the dimension of $T$, the proposition will follow.

Let $  \cR =  \bigoplus_{m\geq 0} \cR_m$ denote the graded $\cO_T$-algebra where $\cR_0 = \cO_T$ 
and  $\cR_{m}= \pi_*  \cO_{X}(mL)$ for $m>0$. Note that  our assumption on the vanishing of higher cohomology implies $\pi_* \cO_X(mL)$ is a vector bundle and commutes with base change for all $m\geq 1$.
In particular, 
$\cR_m \otimes k(\gt)  \simeq H^0(X_{\gt}, \cO_{X_{\gt}}( mL_{\gt}) )$
for each geometric point $\gt \in T$.

Viewing $\cW$ as a subset of $\cR$, let $\cJ \subseteq \cR$ denote the homogeneous ideal generated by $\cW$. 
Note that  
\[
\cJ^m \cap \cR_m   = \im ( S^m (\cW) \to \pi_* \cO_X( mL )).\]
Hence,  for each $\gt \in T$ and $m>0$, 
we have
\begin{equation}\label{eq:JRt}
(W_{\gt})_m : = \im  \left( (\cJ^m \cap \cR_m )\otimes k(\gt) \to  H^0(X_{\gt}, \cO_{X_{\gt}}( mL_{\gt}) ) \right).
\end{equation}

Now,
consider the graded $\cO_T$-algebra given by $\displaystyle \gr_\cJ  \cR :=  \bigoplus_{m \in \N}\cJ^m / \cJ^{m+1}$. Since $\gr_{\cJ}\cR$ is a finitely generated $\cO_T$-algebra, we may apply generic flatness to find a dense open set $U\subseteq T$ so that $\gr_{\cJ} \cR \vert_U $ is flat over $U$. 
Applying the following lemma gives that $(\cJ^m \cap \cR_m )\vert_U$ is flat over $U$ and the natural map 
\begin{equation}\label{J^mrest}
(\cJ^m \cap \cR_m) \otimes k(\gt) \to \cR_m \otimes k(\gt) \simeq H^0(X_{\gt}, \cO_{X_{\gt}}(mL_{\gt}))
\end{equation}
is injective for all $m \geq 0$ and $t\in U$.
Since  $(\cJ^m \cap \cR_m )\vert_U$ is flat over $U$,  $U \ni \gt \mapsto \dim( (\cJ^m \cap \cR_m)\otimes k(\gt) )$ is constant. Therefore, $HF_{W_{\gt}}$ is independent of $\gt\in U$. 
\end{proof}

\begin{lemma}
Let $A\to B$ be a flat morphism of Noetherian rings, $I \subseteq B$ an ideal and $M$ an $A$-module.
If the graded ring $\gr_I B = \bigoplus_{m \geq 0} I^m/I^{m+1}$ is flat over $A$, then for each $m\geq 0$
\begin{itemize} 
\item[(1)] $I^m$, viewed as a $B$-module, is flat over $A$ and
\item[(2)] the natural map $I^m \otimes_A M \to B\otimes_A M$  is injective.
\end{itemize} 
\end{lemma}
\begin{proof}
The statement is trivial for $m=0$. Now, consider the short exact sequence
\begin{equation}\label{eq:Im}
0\to I^{m+1} \to I^{m} \to I^m / I^{m+1}  \to 0,
\end{equation}
and assume the statement holds for $I^m$. Since the latter two terms of \eqref{eq:Im} are flat over $A$, so is $I^{m+1}$. By the flatness of $I^m / I^{m+1}$, \ 
 \eqref{eq:Im} remains exact after applying $\otimes_A M$ and we have
 \[
 0 \to I^{m+1} \otimes_A M \to I^{m} \otimes_A M \to I^{m}/ I^{m+1} \otimes_A M \to 0 \]
Thus, the injectivity of $I^m \otimes M \to B\otimes M$   implies the injectivity of $I^{m+1} \otimes M \to B\otimes_A M$.  
\end{proof}

We are now ready to prove Theorem \ref{thm:finiteW}. 

\begin{proof}[Proof of Theorem  \ref{thm:finiteW}]
It suffices to show that the set
\[
\bigcup_{t\in T}
\big\{ HF_{W}:\N \to \N \, \vert \, W \subseteq H^0(X_{\gt}, \cO_{X_{\gt}}(L_{\gt}))  \text{ with } \dim(W) =r \big\}
\]
is finite for each $r\leq \rank \pi_* \cO_{X}(L)$. Hence, we consider
the Grassmannian $\rho: \G(r,\pi_* \cO_X(L))\to T$  parameterizing rank $r$ subvector bundles of $\pi_* \cO_X(L)$. Note that there is a correspondence between $k(\gt)$-valued points of $\G(r,\pi_* \cO_X(L))_{\gt}$ and rank $r$ subspaces $W\subseteq H^0(X_{\gt}, L_{\gt})$.

Set $X' :=  \Grass(r,\pi_*\cO_X(L)) \times_T X$ and let $\pi'$ and $\rho'$ denote the projection maps
\begin{center}
\begin{tikzcd}
  X'   \arrow[r, "\rho'"] \arrow[d, "\pi'"] & X \arrow[d,"\pi"] \\
  \Grass(r,\pi_*\cO_X(L)) \arrow[r,"\rho"] & T. 
\end{tikzcd}
\end{center}
Write $\cW_{\univ} \subseteq \rho^* (\pi_* \cO_X(L))$ for the universal sub-bundle of the Grassmannian. 
For a geometric point $\gs \in \Grass(r, \pi_* \cO_X(L))$, set
$$
W_{\gs} := 
 \im \left( \cW_{\univ}\otimes k(\gs) \to H^0(X_{\gs},  \cO_{X_{\gs}}(L_{\gs} ) ) \right)
.$$
To complete the proof, it is sufficient to show that
\begin{equation}\label{e-HFfinite}
\{ HF_{W_{\gs}}:\N \to \N \, \vert \,  \gs\in \Grass(r,\pi_* \cO_X(L))  \}
\end{equation}
is finite.

Set $L' = \rho'^*L$. By flat base change, $R^i \pi'_* \cO_{X'} (mL') \simeq  \rho^* R^i \pi_* \cO_X(mL)$ for all $i,m\geq 0$. Hence, our assumption that $R^i \pi_* \cO_X(mL)= 0 $ for all $i,m \geq1$ implies $R^i \pi'_* \cO_{X}(mL')=0$ for all $i,m\geq1$. Additionally, since $\pi'_* \cO_{X'}( L') \simeq \rho^* \pi_* \cO_X(L)$, we may   view $\cW$ as a sub vector bundle of $\pi'_* \cO_{X'}( L')$. Therefore, we may apply Proposition  \ref{prop:Wchain}  to see that \eqref{e-HFfinite} is a finite set.  
\end{proof}

The following corollary  will be used in in the proof of Theorem \ref{t:SSm}.

\begin{corollary}\label{c:volconv}
Keep the setup of Theorem \ref{thm:finiteW}. 
For any $\e>0$, there exists a positive integer $m_0=m_0(\e)$ so that the following holds:
If $t\in T$ and $W\subseteq H^0(X_{\gt},\cO_{\Xgt}(L_{\gt}))$, then 
\[
\left| 
\frac{\vol(W_\bullet)}{\vol(L_{\gt})} - \frac{ \dim(W_m)}{h^0( \cO_{X_{\gt}}(mL_{\gt}) )} \right| < \e
\]
for all $m$ divisible by $m_0$. (The integer $m_0$ is independent of the choice of $t$ and $V$.)
\end{corollary}

\begin{proof}
Given any $t\in T$ and $W\subseteq H^0(\Xgt,\cO_{\Xgt}(L_{\gt}))$,  
\[
\lim_{m \to \infty} \left( \frac{\dim HF_W(m) }{h^0(\Xgt, \cO_{\Xgt}(mL_{\gt}))} \right) = \frac{ \vol(W_\bullet)}{ \vol(L_{\gt})}, \]
since 
\[
\lim_{m \to \infty}  \frac{ \dim HF_W(m) }{m^n/n!} = \vol(W_\bullet) \text{ and }  
\lim_{m \to \infty}  \frac{ h^0(X_{\gt}, \cO_{X_{\gt}}(mL_{\gt}))}{m^n/n!}= \vol(L_{\gt}). \]
The result now follows from  Theorem \ref{thm:finiteW}  and
the fact that $h^0(X_{\gt}, \cO_{\Xgt}(mL_{\gt}))$ is independent of $t \in T$.
\end{proof}

\subsection{Approximation result for $S$ and $T$}\label{ss:approxST}

In this subsection, we state the following approximation result for $S$ and $T$ from \cite{BJ17}, which holds on a fixed variety, and then deduce a result for families.  

\begin{theorem}[\cite{BJ17}]\label{t:fujita}
Let $(X,\Delta)$ be an $n$-dimensional projective klt pair and $L$ an ample Cartier divisor on $X$. There exists 
a positive constant $C$ such that  
\[
0 \leq T(L;v) -T_m(L;v) \leq \frac{C A_{X,\D}(v)}{m} \hspace{.3 in}
\text{ and }
\hspace{.3 in}
0 \leq S(L;v) - \tilde{S}_m(L;v)  \leq \frac{C  A_{X,\D}(v)}{m}\]
for all $m \in \Z_{>0}$ and $v\in \Val_{X}^*$ with $A_{X,\D}(v)< \infty$. 

Furthermore, fix $r\in \N$  such that $r(K_X+\Delta)$ is a Cartier divisor. 
If $b,c \in \Z_{>0}$ are chosen so that 
\begin{itemize}
\item[(1)] $\cO_{X}(cL)$ is globally generated, 
\item[(2)] $bL-K_X-\Delta$ is big and nef, and
\item[(3)] $H^0(X, \cO_X( (c+nb)L) \otimes  \Jac_X \cdot \cO_{X}(-r\Delta ) ) \neq 0$,
\end{itemize}
then the result holds with $C= 1+  (n+1)(c+ nb)/ \alpha(X,\D; L) $.
\end{theorem}

In the case when $\Delta\neq 0$, the above result is proved in \cite[Section 5.4.3]{Blu18}.
\footnote{We note that there is a minor error in \cite[Section 5.4.3]{Blu18} concerning the value $C$.
Page 78 lines 14-19 of \cite{Blu18} should read ``
\[
=\left( 1- \left(\frac{m-a}{m}\right)^{n+1} \right) S(v) + \left( \frac{m-a}{m}\right)^n \frac{A(v)}{m} \leq \frac{a(n+1)}{m} S(v) + \frac{A(v)}{m}
\]
where the last inequality uses that $1-t^{n+1} \leq (n+1)(1-t)$ for $t\in [0,1]$. Since, $S(v) \leq A(v)/ \alpha(L)$ by (5.6) and Corollary 5.3.2. Therefore,
$$ 0 \leq S(v) - \tilde{S}_m(v) \leq \frac{CA(v)}{m}$$ with $C=1+a(n+1)/ \alpha(L)$.''

This error also appeared in an early version of \cite{BJ17}, but is corrected in the published version.}

\begin{corollary}\label{c:STfam}
Let $\pi:(X,\Delta)\to T$ be a projective family of klt pairs and $L$ a $\pi$-ample Cartier divisor on $X$. There exists 
a positive constant $C$ such that the following holds: for each $t\in T$  
\[
0 \leq T(L_{\gt};v) -T_m(L_{\gt};v) \leq \frac{C A_{X_{\gt},\D_{\gt}}(v)}{m} \hspace{.3 in}
\text{ and }
\hspace{.3 in}
0 \leq S(L_{\gt};v) - \tilde{S}_m(L_{\gt};v)  \leq \frac{C  A_{X_{\gt},\D_{\gt}}(v)}{m}\]
for all $m \in \Z_{>0}$ and $v\in \Val_{X_{\gt}}^*$ with $A_{X_{\gt},\D_{\gt}}(v)< \infty$. 
\end{corollary}

The corollary follows easily from the above theorem and the following lemma. 

\begin{lemma}\label{lem:jacnonv}
Keep the assumptions of Theorem \ref{c:STfam}. 
For any $r \in \Z_{>0}$,  there exists a positive integer $m_0=m_0(r)$ so that 
\[
H^0(\Xgt,  \cO_{\Xgt}(mL_{\gt}) \otimes \Jac_{X_{\gt}} \cO_{X}(-r\Delta _{\gt}) ) \neq 0\]
for all $m \geq m_0$ and $t\in T$. 
\end{lemma}

\begin{proof}
Since $L$ is $\pi$-ample, there exists $m_0$ so that 
\[
\pi^*   \pi_* \left(  \cO_X(mL) \otimes \Jac_{X/T}   \cO_{X}(-r \Delta)  \right) )  \to  \cO_{X}(mL) \otimes \Jac_{X}  \cO_{X}(-r\Delta ) 
\]
is surjective for all $m\geq m_1$. Hence, 
\[
H^0(X_{\gt},  ( \cO_{X}(mL) \otimes \Jac_{X/T}    \cO_{X}(-r\Delta ) ) \vert_{X_{\gt}})  \neq 0 
\]
for all $m \geq m_0$ and $t\in T$. Since $\cO_{X}(-r\Delta ) \cdot \cO_{X_{\gt}} \subseteq \cO_{X_{\gt}}(-r\Delta _{\gt})$ and  $\Jac_{X/T}\cdot \cO_{X_{\gt}}= \Jac_{X_{\gt}}$, the result follows. 
\end{proof}

\begin{proof}[Proof of Corollary \ref{c:STfam}]
We seek to find positive integers $b,c$ so that (i)-(iii) of Theorem \ref{t:fujita} are satisfied for each $t\in T$. First, fix  $r\in \N$ so that $r(K_{X/T}+\Delta)$ is a Cartier divisor and apply Lemma \ref{lem:jacnonv} to find $m_0$ so that 
\[
H^0(X_{\gt}, \cO_{X_{\gt}} (mL_{\gt}) \otimes ( \Jac_{X_{\gt}} \cdot \cO_{X_{\gt}}(-r\Delta _{\gt})) ) \neq 0\]
for all $m\geq m_0$ and $t\in T$.
Since $L$ is $\pi$-ample, we may find $b,c \in \Z_{>0}$ so that $$cL_{\gt}
\quad \quad \text{ and } \quad \quad bL_{\gt} - K_{X_{\gt}} - \Delta_{\gt}$$
are very ample for all $t\in T$ and $c+ n b \geq m_0$,  where $n:=\dim(X)-\dim(T)$. 
 Next, set $\gamma := \inf_{t\in T} \alpha(X_{\gt},\Delta_{\gt};L_{\gt})$, which is $>0$  by  Proposition \ref{p:alphabound}. 
Theorem \ref{t:fujita} now implies that the desired inequalities will hold with   $C:=1+(n+1)(c+nb ) /\gamma $. 
\end{proof}

\subsection{A refined approximation result for $S$}\label{ss:refinedapprox}
In the previous subsection, we proved an approximation result for $S$ in terms of $\tilde{S}_m$. 
To prove Theorem \ref{t:dk-s}, we will need the following approximation result for $S$ in terms of $S_m$.

\begin{theorem}\label{t:SSm}
Let $\pi: (X,\Delta) \to T$ be a projective $\Q$-Gorenstein family of klt pairs and $L$  a $\pi$-ample Cartier divisor on $X$. 
For any $\epsilon>0$, there exists a positive integer $m_0:= m_0(\e)$ such that the following holds: 
For each $t\in T$, 
\[
 S(L_{\gt};v)  - S_{  m}(L_{\gt}; v) \leq \varepsilon A_{X_{\gt},\Delta_{\gt}}(v) 
\]
for all positive integers $m$ divisible by $m_0$ and $v\in \Val_{X_{\gt}}$ with $A_{X_{\gt},\Delta_{\gt}}(v)< \infty$. 
\end{theorem}

The statement will eventually be deduced from results in Sections \ref{ss:finitenesshilbert} and \ref{ss:approxST}.

\begin{proposition}\label{p:SSmbir}
Keep the hypotheses of Theorem \ref{t:SSm}. There exists a positive constant $D$ so that
 the following holds:  
If $t\in T$ and $v\in \Val_{X_{\gt}}$ with $A_{X_{\gt},\Delta_{\gt}}(v) < \infty$,
then 
\[
|V_{m}^s|  
\text{ 
 is birational when }
   m \geq 1   \text{ and } 
0 \leq s \leq T(v) - \frac{D A_{X_{\gt},\Delta_{\gt}}}{m}.\]
Here, $V_{m}^s$ is abbreviated notation for  the linear series $\cF_v^{ms}H^0(X_{\gt}, \cO_{X_{\gt}}(mL_{\gt}))$.
\end{proposition}

 To prove the result, we use Corollary \ref{c:STfam} and an argument from the proof of \cite[Lemma 1.6]{BC11}.

\begin{proof}
Fix $C>0$ satisfying the conclusion of Corollary  \ref{c:STfam}, and set $\gamma:=\inf_{t\in T} \alpha(X_{\gt},\Delta_{\gt};L_{\gt})$, which is $>0$  by Proposition \ref{p:alphabound}. 
Since $L$ is $\pi$-ample, we may find $a \in \Z_{>0}$ so that $\cO_{X_{\gt}}(aL_{\gt})$ is very ample for all $t\in T$. 

We claim that the proposition is satisfied with $D:= C+ a/\gamma$. 
Indeed,
fix $t\in T$ and  $v\in \Val_{X_{\gt}}$ with ${A_{X_{\gt}, \Delta_{\gt}}(v) <+\infty}$. 
We will proceed to show  $ |V_m^s|$ is birational when $0 \leq s \leq T(v) - \frac{ D A(v)}{m}$.

For $ m \leq a$, the statement is vacuous. Indeed,  since  $\gamma \leq  \alpha(X_{\gt},\Delta_{\gt};L_{\gt}) \leq A(v)/ T(v)$, we see $T(v) - D A(v)/m <0$ when $m\leq a$. 
 For $m>a$, consider the inclusion:
\begin{equation}\label{eq:Vincl}
V_a^0  \cdot V_{m-a}^{ s m/(m-a)} \subseteq V_m^s .
\end{equation}
Note that $|V_a^0|$ is birational, since  $V_a^0 =H^0(X_{\gt}, \cO_{X_{\gt}}(aL_{\gt}) )$ and $aL_{\gt}$ is very ample. Therefore, inclusion
\eqref{eq:Vincl} implies $V_m^s$ is birational as long as $V_{m-a}^{ sm / (m-a)}$ is nonzero, which is equivalent to the condition $sm/(m-a) \leq T_{m-a}(v)$. Therefore, it is sufficient to show 
\[
s \left( \frac{m}{m-a}\right) \leq T_{m-a}(v) \quad \text{ whenever } m> a \text{ and } 0 \leq s \leq T(v) - \frac{D A(v)}{m}.\]
The latter statement holds, since
\begin{align*}
  \left(T(v) -  \frac{ D A(v)}{m}  \right) \left( \frac{m}{m-a}\right) 
 & = T(v) + \left( \frac{a}{m-a} \right) T(v)  -    \frac{ (a/\gamma + C) A(v)}{m-a}\\
  & \leq  T(v) + \left( \frac{a}{m-a} \right) \frac{A(v)}{\gamma} -  \frac{ (a/\gamma + C) A(v)}{m-a}\\
    &\leq T(v) -   \frac{ CA(v)}{m-a} \\
& \leq T_{m-a}(v),
\end{align*}
where the first inequality follows from the fact that $\gamma \leq \alpha(X_{\gt},  \Delta_{\gt};L_{\gt}) \leq A(v) / T(v)$
and the third from our choice of $C$. Hence, the proof is complete.
\end{proof}

\begin{proposition}\label{p:SSm}
Keep the hypotheses of Theorem \ref{t:SSm}. 
There exists a positive constant $E$ so that
 the following holds:  
If $t\in T$ and $v\in \Val_{X_{\gt}}$  with $A_{X_{\gt},\Delta_{\gt}}(v)<\infty$, then 
\[ 
S(v) \leq \frac{1}{\vol(L_{\gt})} \int_0^{T(v)} \frac{\vol(V_{m,\bullet}^s) }{m^n} \, ds + \frac{E A_{X_{\gt},\Delta_{\gt}}(v)}{m}
\]
for all  $m\in \Z_{>0}$.  
\end{proposition}

\begin{remark}
 Let us explain notation appearing in the above statement and the following proof.
The linear system $V_{m}^s$ is abbreviated notation for   $\cF_v^{ms} H^0(X_{\gt}, \cO_{X_{\gt}} (mL_{\gt}))$. 
 Following Example \ref{ex:lingls}, we write $V^s_{m,\bullet}$ and $\widetilde{V}^s_{m,\bullet}$ for the graded linear series of $mL_{\gt}$ defined by 
 \[
V^s_{m,p} := \im  \left( S^p V^s_{m} \to H^0(X_{\gt}, \cO_{X_{\gt}}( pmL_{\gt} )) \right) 
 \]
 and 
 \[
 \widetilde{V}^s_{m,p}   := H^0(X_{\gt}, \cO_{X_{\gt}} (mL_{\gt}) \otimes \fb_p),
 \]
 where $\fb_p$ is the integral closure of the $p$-th power of the base ideal of $ V_{m}^s$.
\end{remark}

\begin{proof}
Fix positive integers $C$ and $D$ satisfying the conclusions of Corollary \ref{c:STfam} and Proposition \ref{p:SSmbir}. We will show that the  proposition holds with $E= C+D$.

Fix  $t\in T$ and  $v\in \Val_{X_{\gt}}$ with $A(v) <+\infty$.
Observe that for $m>0$
\begin{align*}
S(v) & \leq \widetilde{S}_m(v) +  \frac{C A(v)}{m} \\
&= \frac{1}{\vol(L_{\gt})} \int_{0}^{T(v)} \left( \frac{ \vol( \widetilde{V}_{m,\bullet}^s )}{m^n} \right) \, ds + \frac{C A(v)}{m}  \\
& \leq  \frac{1}{\vol(L_{\gt})} \int_{0}^{  T(v) -
DA(v)/m } \left( \frac{ \vol( \widetilde{V}_{m,\bullet}^s )}{m^n} \right) \, ds + \frac{CA(v)}{m}+  \frac{D A(v)}{m}
\intertext{since $\vol( \tilde{V}_{m,\bullet}^s) \leq m^n \vol(L_{\gt})$.
Since  $\vol(V_{m,\bullet}^s) = \vol(\widetilde{V}_{m,\bullet}^s)$  for all $0 \leq s \leq T(v) - \frac{ D A(v)}{m}$ by our choice of $D$ and Proposition \ref{p:volbir}, the right-hand side above equals} 
&   =  \frac{1}{\vol(L_{\gt})} \int_{0}^{  T(v) - D A(v)/m } \left( \frac{ \vol( {V}_{m,\bullet}^s ) }{m^n} \right) \, ds  +\frac{(C+D) A(v)}{m} \\
&\leq \frac{1}{\vol(L_{\gt})} \int_{0}^{  T(v) } \left( \frac{ \vol( {V}_{m,\bullet}^s  ) }{m^n} \right) \, ds  + \frac{(C+D) A(v)}{m},
\end{align*} and the statement holds with $E=C+D$.
\end{proof}

We will now deduce  Theorem \ref{t:SSm} from the previous proposition and Corollary \ref{c:volconv}.

\begin{proof}[Proof of Theorem \ref{t:SSm}]
After replacing $L$ with a multiple, we may assume $R^i \pi_* \cO_{X}(mL)=0$ for all $i, m \geq1$. 
Now, fix positive integer $E$ satisfying 
the conclusion of Proposition \ref{p:SSm}.
Set \[p: = \lceil 2 E / \varepsilon \rceil
\quad \text{ and }
\quad
\varepsilon' : = \varepsilon \gamma /2,\]
where
$\gamma := \inf_{t\in T} \alpha(X_{\gt},\D_{\gt};L_{\gt})$.

By Corollary \ref{c:volconv}, we may find a positive integer $m_1$ so that the following holds:
 if $t\in T$ and $V_\bullet $ is a graded linear series of $pL_{\gt}$, then 
\begin{equation}\label{eq:volapprox}
\left| 
\frac{\vol(V_{p,\bullet})}{\vol( p L_{\gt})} - \frac{ \dim(V_{p,m})}{h^0( \cO_{X_{\gt}}( m p L_{\gt}))} \right| <  \varepsilon'
\end{equation}
for all positive integers $m$ divisible by $m_1$.

Fix $t\in T$ and a valuation $v\in \Val_{X_{\gt}}$ such that $A(v)< \infty$. 
We claim that if $m$ is a positive integer divisible by $m_1$, then
\[
S(v) \leq S_{ m p}(v) + \varepsilon A(v). 
\]
To see this, observe that 
\begin{align*}
  S(v) &\le \biggl( \frac{1}{\operatorname{vol}(L_{\gt})} \biggr) \int_0^{T(v)} \biggl( \frac{\operatorname{vol}(V^s_{p,\bullet})}{p^n} \biggr)\,ds + \frac{E A(v)}{p}\\
  &\le  \int_0^{T(v)} \biggl( \frac{\dim(V^s_{p, m})}{h^0(\mathcal{O}_{X_t}(pmL_{\gt}))} \biggr)\,ds + \varepsilon' T(v) + \frac{\varepsilon A(v)}{2}\\
  &\le  \int_0^{T(v)} \biggl( \frac{\dim(V^s_{p,m})}{h^0(\mathcal{O}_{X_t}(pmL_{\gt}))} \biggr)\,ds + \frac{\varepsilon A(v)}{2}+  \frac{\varepsilon A(v)}{2}\\
  &\le  \int_0^{T(v)} \biggl( \frac{\dim(V^s_{pm})}{h^0(\mathcal{O}_{X_{\gt}}(pmL_t))} \biggr)\,ds + \varepsilon A(v) \\
  & = S_{pm }(v) + \varepsilon A(v), 
\end{align*}
where  the first inequality follows from our choice of $E$, the second from \eqref{eq:volapprox}, the third from the fact that $T(v) \le A(v)/\alpha(X_{\gt},\Delta_{\gt},L_{\gt}) \le A(v) / \gamma$,
  and the fourth from the inclusion $V^s_{p,m} \subseteq V^s_{pm}$.
Therefore, the result holds with $m_0:=m_1p$.
\end{proof}

\subsection{Proofs of Theorems  \ref{t:alphaconv} and \ref{t:deltaconv}}

\begin{proof}[Proof of Theorem \ref{t:alphaconv}]
We claim that for any $\e>0$, there exists $m_0$ so that 
\[
0\leq \alpha(X_{\gt},\D_{\gt};L_{\gt})^{-1}- \alpha_m(X_{\gt},\D_{\gt};L_{\gt})^{-1}
\leq \e 
\]
for all $t\in T$ and $m$ divisible by $m_0$. Since $T \ni t\mapsto \alpha(X_{\gt}, \D_{\gt};L_{\gt})$ is bounded from above thanks to Proposition \ref{p:alphabound}, the above claim implies the theorem.

To prove the claim, fix a positive constant $C$ satisfying the conclusion of Corollary \ref{c:STfam}. Now, consider $t\in T$. For $v\in \Val_{X_{\gt}}^*$ with $A(v)<+\infty$, our choice of $C$ implies
\[
0\leq 
\frac{T(v)}{A(v)}- \frac{T_m(v)}{A(v)} \leq \frac{C}{m}. 
\]
Combining the previous inequality with Propositions \ref{p:alphamvals} and \ref{p:alphavals}  gives
\[
0 \leq \alpha(X_{\gt},\D_{\gt};L_{\gt})^{-1} - \alpha_m(X_{\gt},\D_{\gt};L_{\gt})^{-1} \leq C/m
.\]
Therefore, the claim holds when $m_0 = \lceil C/ \e \rceil$.
\end{proof}

\begin{proposition}\label{p:ddmconv}
Let $\pi:(X,\D)\to T$ be a projective $\Q$-Gorenstein family of klt pairs over a normal base  and $L$ a $\pi$-ample Cartier divisor on $X$. For $\e>0$, there exists an integer $m_0=m_0(\e)$ such that 
\[
\delta(X_{\gt},\D_{\gt};L_{\gt})^{-1}
-
\delta_m(X_{\gt},\D_{\gt};L_{\gt})^{-1}
\leq  \e 
\]
for all positive integers $m$ divisible by $m_0$ and $t\in T$. 
\end{proposition}

\begin{proof}
Fix $\varepsilon>0$ and 
choose an integer $m_0=m_0(\varepsilon)$ satisfying the conclusion of Theorem \ref{t:SSm}.
For  $t\in T$ and $v\in \Val_{X_{\gt}}^*$ with $A(v)<+\infty$, we 
have
\[
\frac{S(v)}{A(v)} \leq  \frac{S_m(v)}{A(v)} + \e  
\]
for all positive integers $m$ divisible by $m_0$. 
Combining the previous inequality with Proposition \ref{p:dmvals} and Theorem \ref{prop:deltavals} gives 
\[
 \d(X_{\gt},\D_{\gt};L_{\gt})^{-1} \leq  \d_m(X_{\gt},\D_{\gt};L_{\gt})^{-1} + \e.
\]
for all positive integers $m$ divisible by $m_0$ and the proof is complete. 
\end{proof}

\begin{proof}[Proof of Theorem \ref{t:deltaconv}]
We claim that for any $\e>0$, there exists $m_0=m_0(\e)$ so that 
\[
\widehat{\delta} (X_{\gt},\D_{\gt};L_{\gt})^{-1}- \widehat{\d}_{m}(X_{\gt},\D_{\gt};L_{\gt})^{-1}
\leq \e 
 \]
 for all $t\in T$ and positive integers $m$ divisible by $m_0$.
 Since  $T \ni t\mapsto \delta(X_{\gt}, \D_{\gt};L_{\gt})$
is bounded from above (see Propositions \ref{p:alpdelin} and \ref{p:alphabound}), the above claim implies the proposition.

To prove the claim, apply Proposition \ref{p:ddmconv} to choose an integer $m_1$ so that 
\begin{equation}\label{e:ddmconv}
\delta (X_{\gt},\D_{\gt};L_{\gt})^{-1}- {\d}_{m}(X_{\gt},\D_{\gt};L_{\gt})^{-1}
\leq \e/2 
\end{equation}
 for all $t\in T$ and $m$ divisible by $m_1$.
 Combining \eqref{e:ddmconv} with  Proposition \ref{p:dmdmhat}, we see
 \[
 \widehat{\delta} (X_{\gt},\D_{\gt};L_{\gt})^{-1}- \widehat{\d}_{m}(X_{\gt},\D_{\gt};L_{\gt})^{-1}
\leq \e/2 +m^{-1}\a(X_{\gt},\D_{\gt};L_{\gt})^{-1}
 \]
 for all $t\in T$ and $m$ divisible by $m_1$.
 Thanks to Proposition \ref{p:alphabound}, there exists a positive integer $m_2$ so that $$m^{-1}\a(X_{\gt},\D_{\gt};L_{\gt})^{-1}<\e/2$$ for all $t\in T$ and positive integers $m$ divisible by $m_2$. Hence, the desired statement holds with $m_0=m_1 \cdot m_2$. 
\end{proof} 

\section{Lower Semicontinuity Results} 

\subsection{Lower semicontinuous functions}

Recall that a function $f:X\to \R$, where $X$ is a topological space, is lower semicontinuous iff $\{ x\in X \, \vert \, f(x) > a \}$ is open for every $a \in \R$. The following elementary real analysis result will be used  to show that our thresholds  are lower semicontinuous in families.  

\begin{proposition}\label{p:uniformlsc}
Let $X$ be a topological space and $(f_m:X\to \R)_{m\in \N}$  a sequence of functions converging pointwise to a function $f:X\to \R$
such that:
\begin{itemize}
    \item[(1)] For $m$ sufficiently divisible,  $f_m$ is lower semicontinuous;
    \item[(2)] For each $\e>0$, there exists  a positive integer $m_0:= m_0(\e)$ so that for each $x\in X$
    \[
    f_m(x)\leq f(x)+\varepsilon \quad \quad \text{ for all $m$ divisible by $m_0$}.\]
\end{itemize}
Then $f$ is lower semicontinuous. 
\end{proposition}

\subsection{Semicontinuity of the global log canonical threshold}

\begin{proposition}\label{p:alsc}
Let $\pi: (X,\Delta) \to T$ be a projective $\Q$-Gorenstein family of klt pairs  over a normal base  and $L$  a $\pi$-ample Cartier divisor on $X$.  For $m\gg0$, 
  the function $T\ni t \mapsto  \alpha_m(X_{\overline{t}}  ,\D_{\overline{t}} ;L_{\overline{t}}) $ is lower semicontinuous and takes finitely many values. 
  \end{proposition}
  
  \begin{proof}
  Fix $m\gg0$, so that $R^i \pi_* \cO_X(mL) =0$ for all $i>0$.  Hence, for such $m$, $\pi_*\cO_X(mL)$ is a vector bundle and $\pi_* \cO_{X}(mL)$ commutes with base-change. 
  
  Consider the projective bundle $\rho: W = \mathbb{P}( \pi_* \cO_{X}(mL)^*) \to T$. For $t\in T$, we have a bijection between $k(\gt)$-valued points of $W_{\gt}$ and $D\in | mL_{\gt}|$. Let $\Gamma$ be the universal divisor on $W\times_T X$ with respect to this correspondence. 
  
  By \cite[Lemma 8.10]{KP17}, the function $W\ni \gy \mapsto \lct( X_{\gy}, \D_{\gy};m^{-1}\Gamma_{\gy}) $ is lower semicontinuous and takes finitely many values. Hence, there exists finitely many rational numbers $a_1>a_2> \cdots >a_s $ and a sequence of closed sets 
  \[
  W = Z_1 \supsetneq Z_2 \supsetneq \cdots \supsetneq Z_{s}
\supsetneq Z_{s+1} = \emptyset \]
such that if $\overline{y} \in Z_{i}\setminus Z_{i+1}$, then $\lct(X_{\gy}, \D_{\gy};m^{-1}\Gamma_{\gy}) = a_i$. 
Therefore,  $\{\alpha_m(\Xgt,\D_{\gt};L_{\gt}) \, \vert t\in T\} \subseteq \{a_1,\ldots, a_s \}$. 

To prove the lower semicontinuity of $\alpha_m$, it suffices to show  \[
\{ t\in T \, \vert \, \alpha_m(\Xgt,\D_{\gt};L_{\gt}) \leq a_i \}
\]
is closed
for each $i \in \{1,\ldots, s\}$.
Now, observe that $ t\in \rho(Z_i)$ iff $(Z_i)_{\gt}$ contains a $k(\gt)$-valued point.  Therefore, 
\[
\rho(Z_i) = \{ t\in T \, \vert \, \alpha_m(\Xgt,\D_{\gt};L_{\gt}) \leq a_i \}
.\]
Since $\rho$ is proper and $Z_i $ is closed, the latter set is  closed. 
  \end{proof}

  \begin{theorem}\label{t:alphalsc}
If $\pi: (X,\Delta) \to T$ is a projective $\Q$-Gorenstein family of klt pairs over a normal base  and $L$  a $\pi$-ample Cartier divisor on $X$, then the function $T\ni t \mapsto \alpha(X_{\overline{t}}  ,\D_{\overline{t}} ;L_{\overline{t}}) 
$  is lower semicontinuous. 
  \end{theorem}

\begin{proof} The result follows from combining  Theorem \ref{t:alphaconv} and Proposition \ref{p:alsc} with Proposition \ref{p:uniformlsc}.
\end{proof}

\subsection{Semicontinuity of the stability threshold}

  \begin{proposition}\label{p:dmlsc}
Let $\pi: (X,\Delta) \to T$ be a projective $\Q$-Gorenstein family of klt pairs over a normal base  and $L$  a $\pi$-ample Cartier divisor on $X$.  For $m\gg0$, 
  the function $T\ni t \mapsto  \widehat{\delta}_m(X_{\overline{t}}  ,\D_{\overline{t}} ;L_{\overline{t}}) $ is lower semicontinuous and takes finitely many values. 
  \end{proposition}
  
  To approach the above proposition, we seek to  parametrize $\N$-filtrations of $H^0(\Xgt, \cO_{\Xgt}(mL_{\gt}))$ satisfying $T_m(\cF) \leq 1$. Recall that such a filtration is equivalent to the data of a length $m$ decreasing sequence of subspaces of $H^0(\Xgt, \cO_{\Xgt}(mL_{\gt}))$.

 Fix $m\gg0$ so that $R^i \pi_* \cO_X(mL) =0$ for all $i>0$. Set $N_m =\rank(\pi_* \cO_X(mL))$.  Hence, $\pi_*\cO_X(mL)$ is a vector bundle of rank $N_m$ and commutes with base change. For each sequence of integers $\ell = (\ell_1 , \ldots,\ell_m) \in \N^m$ satisfying \begin{equation}\label{e:lcond}
N_m \geq \ell_1 \geq \ell_2 \geq \cdots \geq \ell_m \geq 0, \end{equation}
let
 $\rho_{\ell}:\Fl^{m,\ell}\to T$ denote the relative flag variety for $\pi_* \cO_X(mL)$ that parametrizes flags of signature $\ell$. Hence, for a geometric point $\gt\in T$, there is a bijection between $k(\gt)$-valued points of $\Fl^{m,\ell}_{\gt}$ and $\N$-filtrations $\cF$ of $H^0(\Xgt,\cO_{\Xgt}(mL_{\gt}))$ satisfying 
\[
\dim_{k(\gt)} ( \cF^i  H^0(\Xgt, \cO_{\Xgt}(mL_{\gt}) ) = 
\begin{cases}  
\ell_i  &\text{ for } 1 \leq i \leq m \\
0 & \text{ for } i >m 
\end{cases}.\]
For a geometric point 
$\gy\in \Fl^m$, we write $\cF_{\gy}$ for the corresponding filtration of $H^0(X_{\gy}, \cO_{X_{\gy}}(mL_{\gy}))$.

Let $\Fl^m$ denote the disjoint union $\sqcup_{\ell} \Fl^{m,\ell}$, where the union runs through all $0 \neq \ell \in \N^m$ satisfying \eqref{e:lcond}. Hence, for $t\in T$, there is a bijection between $k(\gt)$-valued points of $\Fl^{m}_{\gt}$ and non-trivial $\N$-filtrations $\cF$ of $H^0(\Xgt,\cO_{\Xgt}(mL_{\gt}))$ satisfying 
$T_m(\cF)\leq 1$.  Let $\rho:\Fl^m \to T$ denote the map induced by the $\rho_{\ell}$'s.

\begin{lemma}\label{l:l/Slsc}
 The function  $\Fl^{m} \ni \gy  \to \lct(\Xgy, \Delta_{\gy}; \fb_\bullet( \widehat{ \cF_{\gy}})) / S_m (\cF_{\gy}) $ is lower semicontinuous and takes finitely many values. 
\end{lemma}

\begin{proof}
Note that $\gy \mapsto S_{m}(\cF_{\gy})$ is constant on each irredicuble component of $\Fl^m$. Indeed, for any $\gy \in \Fl^{m,\ell}$, $
S_m(\cF_{\gy}) = \frac{1}{mN_m}\sum_{i=1}^m \ell_i$. Hence, we are reduced to showing that $\gy \mapsto \lct(\Xgy, \Delta_{\gy};  \fb_\bullet( \widehat{ \cF_{\gy} })$ is lower semicontinuous and takes finitely many values. 

Set $X': = X\times_T \Fl^{m}$, and write $\pi'$ and $\rho'$ for the projection maps:
\begin{center}
\begin{tikzcd}
  X'   \arrow[r, "\rho'"] \arrow[d, "\pi'"] & X \arrow[d,"\pi"] \\
  \Fl^{m} \arrow[r,"\rho"] & T. 
\end{tikzcd}
\end{center}
Set $L' := \rho^* (L)$, and note that $\pi'_* \cO_{X'}(mL') \simeq \rho^* \pi_* \cO_{X}(mL)$ by  flat base change.

On $\Fl^{m}$ there is a universal flag 
\[
\pi'_* \cO_{X'}(mL')
\supseteq \cW_{\univ,1} \supseteq \cdots \supseteq  \cW_{\univ,m} . 
\]
such that $\cF_{\gy}^i H^0(X_{\gy}, \cO_{\Xgy}(mL_{\gy})) $ is the image of the map
\[
\cW_{\univ,i} \otimes k(\gy) \to H^0(X_{\gy}, \cO_{\Xgy}(mL_{\gy}) )
\]
 for each $\gy \in \Fl^{m}$ and $i\in \{1,\ldots,m\}$. The universal flag gives rise a universal sequence of base ideals. Indeed, for each $i \in \{1,\ldots, m\}$, set
\[
\fa_{\univ,i}: = \im \left(
\pi'^*(\cW_{\univ,i}) \otimes \cO_{X'}(-mL') \to \cO_{X'} \right),
\]
where the previous map is induced by the map $\pi'^* \pi'_* \cO_{X'}(mL') \otimes \cO_{X'}(-mL') \to \cO_{X'} $. Note that the base ideal of $\cF^i_{\gy} H^0(\Xgy, \cO_{\Xgy}(mL_{\gy}))$ equals
$ \fa_{\univ,i} \cdot \cO_{\Xgy}$. 

Now, set 
 \[
 \fb_{\univ,p} = \sum_{c} \fa_{\univ,1}^{c_1} \cdots \fa_{\univ,m}^{c_m}\]
 where the sum runs through all $c = (c_1,\ldots,c_m) \in \N^{m}$ such that $\sum i c_i  = p$.
By Lemma \ref{l:extfilt},
\begin{equation}\label{e-baseu}
 \fb_p( \widehat{\cF}_{\gy} ) = \fb_{u,p}\cdot \cO_{\Xgy}
 \end{equation}
 for all $\gy \in \Fl^{m}$ and $p\in \N$.

 Next, apply Lemma \ref{prop:gord} to find  $N\in \Z_{>0}$ so that $\fb_{u,Np}= \fb_{u,N}^p$ for all $p>0$. By \eqref{e-baseu}, this implies $ \fb_{Np}( \widehat{\cF}_{\gy} ) = 
 \fb_{N}( \widehat{\cF}_{\gy} )^p$
 for all $\gy \in \Fl^{m}$ and $p >0$. Therefore, \[
 \lct( \Xgy, \Delta_{\gy};  \fb_{\bullet}( \widehat{\cF}_{\gy} ) ) = N  \lct(\Xgy, \Delta_{\gy};   \fb_{N}( \widehat{\cF}_{\gy} ) ) = N \lct(\Xgy, \Delta_{\gy};  \fb_{u,N} \cdot \cO_{\Xgy})\]
 for all  $\gy \in \Fl^{m}$. 
Hence, it suffices to show
 $\gy \mapsto N \lct( (\Xgy, \Delta_{\gy};  \fb_{u,N} \cdot \cO_{\Xgy} ) $ is lower semicontinuous and takes finitely many values. Since the latter holds by  \cite[Lemma 8.10]{KP17}, the proof is complete.
\end{proof}
 
 \begin{proof}[Proof of Proposition \ref{p:dmlsc}]
Fix $m\gg0$ so that  $R^i\pi_* \cO_X(mL) =0$ for all $i>0$. Hence, $\pi_*\cO_X(mL)$ is a vector bundle and commutes with base change. 

Consider $\Fl^m$ as defined previously.
By Lemma \ref{l:l/Slsc}, there exist finitely many rational numbers $a_1 > a_2 > \cdots > a_s $ and a sequence of closed sets 
\[
\Fl^m = Z_1 \supsetneq  Z_2 \supsetneq  \cdots \supsetneq Z_{s} \supsetneq  Z_{s+1} = \emptyset
\]
such that if $\gy \in Z_{i} \setminus Z_{i+1}$, then $ a_i = \lct(\fb_\bullet(\widehat{\cF}_{\gy}) )/ S_m(\cF_{\gy})$. 
Recall that for $t\in T$, there is a bijection between non-trivial $\N$-filtrations of $H^0(\Xgt, \cO_{\Xgt}(mL_{\gt}))$ and
$k(\gt)$-valued points of $(\Fl^m)_{\gt}$. 
Therefore, $\{\widehat{\delta}_m (X_{\gt},\D_{\gt}; L_{\gt}) \, \vert t\in T \} \subseteq \{ a_1, \ldots, a_s\}$.

To prove the lower semicontinuity of $\delta_m$, it suffices to show
\[
\{ t\in T \, \vert \, \widehat{\delta}_m (X_{\gt},\D_{\gt}; L_{\gt}) \leq a_i\}\]
is closed for each $i\in \{1,\ldots, s\}$.
To proceed, observe that
 $t\in \rho (Z_i)$ iff $(Z_i)_{\gt}$ contains a $k(\gt)$-valued point.
Therefore, 
\[
\rho(Z_i) = \{ t\in T \vert \delta_m(X_{\gt}, \Delta_{\gt}; L_{\gt}) \leq a_i \}.
\]
Since $\rho$ is proper and each $Z_i$ is closed, the latter set is closed.
 \end{proof}

  \begin{theorem}\label{t:deltalsc}
If $\pi: (X,\Delta) \to T$ be a projective $\Q$-Gorenstein family of klt pairs over a normal base  and $L$  a $\pi$-ample Cartier divisor on $X$, then the function $T\ni t \mapsto \delta(X_{\overline{t}}  ,\D_{\overline{t}} ;L_{\overline{t}}) 
$  is lower semicontinuous. 
  \end{theorem}

  \begin{proof}
  The result follows from combining  Theorem \ref{t:deltaconv} and Proposition \ref{p:dmlsc} with Proposition \ref{p:uniformlsc}.
\end{proof}
  
  \begin{remark}
   In \cite[Proposition 4.14]{CP18}, it is shown that the stability threshold is constant on very general points. The result also follows from Theorem \ref{t:deltalsc}.
  \end{remark}

\begin{proof}[Proof of Theorem \ref{t:B}]
The statement is a special case of Theorems \ref{t:alphalsc} and \ref{t:deltalsc}.
\end{proof}

\subsection{Openness of uniform K-stability} 

The following result follows from Theorems \ref{t:dk-s} and \ref{t:deltalsc}. 

\begin{theorem}\label{t:logKopen}
If $(X,\D)\to T$ be a projective $\Q$-Gorenstein family of klt pairs over a normal base such that $-K_{X/T}-\D$ is $\pi$-ample, then 
\begin{itemize}
    \item[(1)] $\{ t\in T \, \vert \, (X_{\gt}, \Delta_{\gt}) \text{ is uniformly K-stable} \}$ is an open subset of $T$, and 
    \item[(2)] $\{ t\in T \, \vert \, (X_{\gt},\Delta_{\gt}) \text{ is K-semistable} \}$ is a countable intersection of open subsets of $T$.
\end{itemize}
\end{theorem}

\begin{proof}
By Theorem \ref{t:deltalsc} with $L := -K_X-\D$, we see 
$$\{ t\in T \, \vert \, \delta(X_{\gt},\D_{\gt};-K_{X_{\gt}}-\D_{\gt})>1 \}
$$
is open in $T$ and 
$$\bigcap_{m\geq 1}\{ t\in T \, \vert \, \delta(X_{\gt},\D_{\gt};-K_{X_{\gt}}-\D_{\gt})>1-1/m \} 
$$
is a countable intersection of open subsets of $T$. 
Applying Theorem \ref{t:dk-s} completes the proof. 
\end{proof}

\begin{proof}[Proof of Theorem \ref{t:A}]
Let $V\subseteq T$ denote the locus of point $t\in T$ such that $X_{\gt}$ is klt. The set $V$ is open in $T$ \cite[Corollary 4.10.2]{Kol13} and contains all K-semistable geometric fibers \cite[Theorem 1.3]{Oda13dis}. Applying Theorem \ref{t:logKopen} to the family $X_V \to V$ with $\Delta=0$ completes the proof.  
\end{proof}

\section{The stability threshold and K-stability for log pairs}

We first give a motivation from complex geometry. For a Fano manifold $X$, the \emph{greatest Ricci lower bound} (or  \emph{$\beta$-invariant}\footnote{The $\beta$-invariant of a \emph{Fano manifold} defined here is different from the $\beta$-invariant of a \emph{divisorial valuation} introduced by Fujita in \cite{Fujitavalcrit}.}) of $X$ is defined as
\[
\beta(X):=\sup\{t\in[0,1]\mid \textrm{ there exists a K\"ahler metric }\omega\in c_1(X)\textrm{ such that }\Ric(\omega)>t\omega\}.
\]
This invariant was studied by Tian in \cite{Tia92}, although it was not explicitly defined there. It was first explicitly defined by Rubinstein in \cite{Rub08, Rub09} and was later further
studied by Sz\'ekelyhidi \cite{Sze11}, Li \cite{Li11}, Song and Wang \cite{SW16}, and Cable \cite{Cab18}. (Note that $\beta(X)$ is denoted by $R(X)$ in some papers.) In the following result, Song and Wang study the relationship between $\beta(X)$ and the existence of conical K\"ahler-Einstein metrics.

\begin{theorem}\cite[Theorem 1.1]{SW16}\label{t:songwang}
Let $X$ be a Fano manifold. 
\begin{enumerate}
    \item For any $\beta\in [\beta(X),1]$ and smooth
    divisor $D\in|-mK_X|$ with $m\in\bN$, there does not exist a smooth conical K\"ahler-Einstein metric $\omega$ with \begin{equation}\label{e:conicalKE} \Ric(\omega)=\beta\omega+\frac{1-\beta}{m}[D]
    \end{equation}
    if $\beta(X)<1$.
    \item For any $\beta\in (0,\beta(X))$, there exists a
    smooth divisor $D\in |-mK_X|$ for some $m\in\bN$
    and a smooth conical K\"ahler-Einstein metric $\omega$ satisfying \eqref{e:conicalKE}.
\end{enumerate}
\end{theorem}

It is shown by Berman, Boucksom and Jonsson \cite{BBJ18}
and independently by Cheltsov, Rubinstein and Zhang \cite{CRZ18}
that $\beta(X)=\min\{1,\delta(X)\}$ for any Fano manifold $X$.

\subsection{
An algebraic analogue of Theorem \ref{t:songwang}.
}

In this section, we prove the following algebraic analogue of Theorem \ref{t:songwang}.  Note that a similar result is proved independently in \cite{CRZ18}.

\begin{theorem}\label{t:betalogk}
Let $(X,\Delta)$ be a log Fano pair.
\begin{enumerate}
    \item For any rational number $\beta\in (\delta(X,\Delta),1]$ and
    any $D\in |-K_X-\Delta|_{\Q}$,
    the pair $(X,\Delta+(1-\beta)D)$ is not  K-semistable when $\delta(X,\D)<1$. Moreover, the pair $(X,\Delta+(1-\beta)D)$ is not uniformly  K-stable when $\beta = \delta(X,\D) \leq 1$.
    \item For any rational number $\beta\in (0,\min\{1,\delta(X,\Delta)\})$, there exists an effective $\bQ$-divisor $D\sim_{\bQ}-(K_X+\Delta)$ such that the pair
    $(X,\Delta+(1-\beta)D)$ is uniformly  K-stable.
\end{enumerate} 
\end{theorem}

\begin{proof}
(1)  Assume $\delta(X,\D)\leq 1$ and fix $\beta\in [\delta(X,\Delta),1]$ and $D\in |-K_X-\Delta|_{\Q}$. If $(X,\Delta+(1-\beta)D)$ is not klt, then  pair is not K-semistable by \cite[Corollary 9.6]{BHJ1}. We move onto the case when  $(X,\Delta+(1-\beta)D)$ is klt.

Fix $v\in \Val_{X}^*$ with $A_{X,\D}(v)< +\infty$. 
Since $-(K_X+\Delta+(1-\beta)D) \sim_{\Q} -\beta(K_X+\D)$, we have
\[
 S(- (K_X+\Delta+(1-\beta)D);v)= \beta S(-K_X-\D ;v)
\]
by Proposition \ref{p:STprops}.2. 
We also have
\[
A_{X,\Delta+(1-\beta)D}(v)=A_{X,\Delta}(v)-(1-\beta)v(D)\leq A_{X,\Delta}(v).
\]
Hence,
\[
\delta(X,\Delta+(1-\beta)D)=\inf_{v}\frac{A_{X,\Delta+(1-\beta)D}(v)}{ S(- (K_X+\Delta+(1-\beta)D);v)}\leq
\inf_{v}\frac{A_{X,\Delta}(v)}{\beta  S(-K_X-\D;v)}=\frac{\delta(X,\Delta)}{\beta}.
\] If $\beta > \delta (X,\D)$, we have  $\delta(X,\Delta+(1-\beta)D)<1$. If $\beta = \delta(X,\Delta)$, then $\delta(X,\Delta+(1-\beta)D)\leq 1$.  Applying  Theorem \ref{t:dk-s}  completes the proof of (1).

(2) Fix $\beta \in (0, \min\{1,\delta(X,\D)  \})$.
Let $m\geq 2$ be chosen so that $-m(K_X+\Delta)$ is a Cartier divisor and the linear system $|-m(K_X+\D)|$ is base point free. Then, for a general $\bQ$-divisor $D\in\frac{1}{m}|-m(K_X+\Delta)|$ the pair $(X,\Delta+mD)$ is klt by \cite[Lemma 5.17]{KM98}. In particular, $
A_{X,\Delta}(v)\geq mv(D)$ for any $v\in\Val_X^*$. 

Consider $v\in \Val_X$ with $A_{X,\D}(v)<+\infty$. We have
\begin{align*}
A_{X,\Delta+(1-\beta)D}(v) & = A_{X,\Delta}(v)-(1-\beta)v(D) \\
    & \geq (1 -(1-\beta)/m) A_{X,\D}(v).
\end{align*}
As in the proof of (1), we also have
 $S(- (K_X+\Delta+(1-\beta)D);v)= \beta S(-K_X-\D ;v)$. 
 Therefore, 
\begin{align*}
\delta(X,\Delta+(1-\beta)D)&\geq \inf_{v} \frac{(1-(1-\beta)/m)A_{(X,\Delta)}}{ \beta S(- (K_X+\Delta;v)}= 
\frac{1- (1-\beta)/m}{\beta}\delta(X,\Delta).
\end{align*}
Thus, if $m$ was chosen sufficiently large and divisible, 
then
$\delta(X,\Delta+(1-\beta)D)>1$. Hence, $(X,\D+(1-\beta)D)$ is uniformly K-stable by Theorem \ref{t:dk-s}. 
\end{proof}

\begin{proof}[Proof of Theorem \ref{t:C}]
The statement follows immediately from Theorem \ref{t:betalogk}.2. 
\end{proof}

The proof of Theorem \ref{t:betalogk}.2  implies
the following result which can be viewed as a K-stability analogue of
\cite[Proposition 1.1]{SW16}.

\begin{proposition}\label{p:SWformula}
Let $(X,\D)$ be a log Fano pair and $m$ an integer $\geq 2$. If $D=m^{-1}H$, where $H\in |-m(K_X+\Delta)|$ satisfies that $(X,\Delta+H)$ is log canonical, then $(X,\Delta+(1-\beta)D)$ is uniformly  K-stable for any $\beta\in (0,\frac{(m-1) \min\{1, \d(X,\D) \} }{m-\min\{1, \d(X,\D) \}})$.
\end{proposition}

In light of Theorem \ref{t:betalogk}.2 and Proposition \ref{p:SWformula},
it is natural to conjecture that the following stronger statement holds. Indeed, such a conjecture can be viewed as a modified version of Donaldson's conjecture \cite[Conjecture 1]{Don12} according to the examples in \cite{Sze13}.

\begin{conjecture}\label{q:donaldson}
Let $(X,\Delta)$ be a log Fano pair that is not K-semistable.
Then there exists ${D \in |-K_X-\Delta|_{\Q}}$ such that 
\[(X, \Delta+(1-\beta)D) \quad  \text{ is uniformly K-stable}
\] for all $0<\beta <\delta(X,\Delta)$.
\end{conjecture}

The next theorem is an  application of Theorem
\ref{t:betalogk}.

\begin{theorem}
Assume the Zariski openness of uniform K-stability in $\bQ$-Gorenstein flat families of log Fano pairs.
Then for any $\bQ$-Gorenstein flat family $\pi:(X,\Delta)\to T$ of log Fano pairs, the function $T\ni t\mapsto \min\{1,\delta(X_{\gt},\Delta_{\gt})\}$  is lower semicontinuous in the Zariski topology.
\end{theorem}

\begin{proof}
It suffices to show that for any rational number
$\beta\in (0,1)$, the locus $\{t\in T\mid \delta(X_{\gt},\Delta_{\gt})>\beta\}$ is Zariski open.
Assume that $\delta(X_{\go},\Delta_{\go})>\beta$ for some point $o\in T$. Then by Theorem \ref{t:betalogk}.2, there
exists an effective $\bQ$-divisor $D_{\go}\sim_{\bQ}-(K_{X_{\go}}+\Delta_{\go})$ such that $(X_{\go},\Delta_{\go}+(1-\beta)D_{\go})$ is uniformly  K-stable.
Let us choose $m\in\bN$ sufficiently divisible such that
$mD_{\go}$ is Cartier, $-m(K_{X/T}+\Delta)$ is Cartier,
and $\pi_*\cO_X(-m(K_{X/T}+\Delta))$ is locally free on $T$. The projective bundle  $W:=\bP_T(\pi_*\cO_X(-m(K_{X/T}+\Delta))^{*})$ over $T$ parametrizes effective $\bQ$-divisors $D_{\gt}\in\frac{1}{m}|-m(K_{X_{\gt}}+\Delta_{\gt})|$ on $X_{\gt}$. 
Since $(X_{\go},\Delta_{\go}+(1-\beta)D_{\go})$ is uniformly K-stable, by the openness of uniform K-stability we can find
an open set $U$ of $W$ containing $D_{\go}$, such that for 
any $D_{\gt}\in U$ the pair $(X_{\gt},\Delta_{\gt}+(1-\beta)D_{\gt})$
is uniformly K-stable. Denote by $\psi:W\to T$ the projection morphism, then $\psi(U)$ is an open neighborhood of $o$ in $T$ since $\psi$ is flat.
Hence part (1) of Theorem \ref{t:betalogk} implies that
$\delta(X_{\gt},\Delta_{\gt})>\beta$ for any $t\in \psi(U)$.
\end{proof}

\begin{remark}
 Using the weak openness of K-semistability from \cite{BL18} and Theorem \ref{t:betalogk}, the above proof implies the  weak lower semicontinuity of $T\ni t\mapsto\min\{1,\delta(X_{\gt},\Delta_{\gt})\}$.
\end{remark}

\subsection{The toric case}

In this section, we will explain that a stronger version of Theorem \ref{t:betalogk} holds in the toric setting.
Specifically, we confirm Conjecture \ref{q:donaldson} when $(X,\Delta)$ is a toric log Fano pair.

\subsubsection{Setup}
Throughout, we will freely use results and notation from ~\cite{Ful93} for toric varieties.
Fix a projective toric variety $X=X(\Sigma)$ given by a rational fan $\Sigma \subset N_\R$, where $N\simeq \Z^n$ is a lattice and $N_\R := N \otimes_{\Z} \R$.
We write $M = \Hom(N, \Z)$, $M_\Q=M\otimes_\Z\Q$, and $M_\R = M \otimes_\Z \R$ for the corresponding dual lattice and vector spaces. 

Let $v_1, \ldots, v_d$ denote the primitive generators of the one-dimensional 
cones in $\Sigma$ and $D_1, \ldots, D_d$ be the corresponding torus invariant 
divisors on $X$. When the context is clear, we will a bit abusively write $v_i$ for the valuation $\ord_{D_i}$.

Fix  torus invariant $\Q$-divisors 
$$\D= \sum_{i=1}^d b_i D_i
\quad \text{ and } \quad
L = \sum_{i=1}^d c_i D_i.$$
so that (i) $\D$ has coefficients in $[0,1)$, $K_X+\D$ is $\Q$-Cartier, and (ii)  $L$ is $\Q$-Cartier and ample. Assumption (i) implies  $(X,\D)$ is klt.

Associated to $L$ is the convex polytope 
\[
  P_L= \{u \in M_\R \, \vert \, \langle u, v_i \rangle \ge -c_i \text{ for all } 1 \le i \le d \}.
\]
Let $\overline{u} \in M_{\Q}$ denote the barycenter of $P_{L}$. 
Recall that there is a correspondence between points in $P_L \cap M_\Q$ and
effective torus invariant $\Q$-divisors $\Q$-linearly equivalent to $L$, under which 
$u \in P_L \cap M_\Q$ corresponds to 
\[
  D_u 
  := L+ \sum_{i=1}^{d} \langle u , v_i \rangle D_i
  := \sum_{i=1}^{d} (\langle u , v_i \rangle+c_i)D_i.
\]

\subsubsection{The stability threshold}
We recall the following result from \cite[\S 7]{BJ17} (and \cite{Blu18} for the setting of log pairs) on the value of the stability threshold in the toric case.

\begin{proposition}\label{p:Tdelta}
With the above setup,
$$ A_{X,\D}(v_i) = 1- b_i \quad \text{ and } \quad
S(L;v_i) = \langle \overline{u},v_{i} \rangle +c_i$$
for each $i \in \{1,\ldots, d\}$. 
\end{proposition}

\begin{theorem}\label{t:Tdelta}
With the above setup,
$$\delta(X,\D;L) = \min\limits_{i=1,\ldots,d}
\frac{A_{X,\D}(v_i)}{S(L;v_i)}= \min\limits_{i=1,\ldots,d} \frac{1-b_i}{  \langle \overline{u},v_{i} \rangle +c_i}$$
\end{theorem}

\subsubsection{Log Fano toric pairs}
We keep the previous setup, but will additionally assume  $(X,\D)$ is a toric log Fano pair. Hence $ -K_X - \D = \sum_{i=1}^d (1-b_i) D_i$ is ample.
The vector $\overline{u}$ will denote the barycenter of $$
P_{-K_X-\D}:=  \{u \in M_\R \, \vert \, \langle u, v_i \rangle \ge -1+b_i \text{ for all } 1 \le i \le d \}.$$ 

The following statement appeared \cite[\S 7.6]{BJ17} in the $\Q$-Fano case. As we will explain, the more general result follows from the same argument.

\begin{proposition}\label{p:TlFd}
Let $(X,\D)$ be a toric log Fano pair and $\overline{u}$ denote the barycenter of $P_{-K_{X}-\D}$.  
\begin{itemize}
    \item[(1)] If $\overline{u}$ is the origin, then $\d(X,\D)=1$. 
    \item[(2)] If $\overline{u}$ is not the origin, then 
    $$ \delta(X,\D) = \frac{c}{1+c} \in (0,1), $$
    where $c$ is the largest real number such that $-c\overline{u} \in P_{-K_{X}-\D}$. 
\end{itemize}
\end{proposition}

\begin{proof}
 Theorem \ref{t:Tdelta} in the case when $L= -K_X -\D$ gives
\begin{equation}\label{e:Tdelta}
\delta(X,\D) = \min_{i=1,\ldots,d} \frac{1-b_i}{
1-b_i + \langle \overline{u},v_i \rangle }.\end{equation}
Statement (1)  follows immediately from \eqref{e:Tdelta}. 
For (2), we claim that if $\overline{u}$ is not the origin, then
$$0<\langle \overline{u},v_i \rangle + (1-b_i) \leq (1-b_i)/c +(1-b_i) $$
for all $i$ and the last inequality is an equality for some $i$. Statement (2) now follows from the claim and \eqref{e:Tdelta}. 

We now prove the claim. Since $\overline{u}$ lies in the interior of $P_{-K_X-\D}$, $\langle \overline{u},v_i \rangle > -1+b_i$ for all $i$. Since $-c\overline{u}$ lies on the boundary of $P_{-K_X-\D}$, 
$$
-c \langle \overline{u},v_i \rangle =  \langle -c\overline{u},v_i \rangle \geq -1+b_i 
$$
and the last inequality holds for some $i$. This completes the proof. 
\end{proof}

The following  statements are inspired by  results in complex geometry (specifically,  \cite[Theorem 3.3.2]{SW16} and \cite[Theorem 1.14]{LS14}). We thank Song Sun for bringing our attention to the previous results and suggesting the existence of algebraic analogues.
 
\begin{proposition}\label{l:T*delta}
Let $(X,\D)$ be a toric log Fano pair that is not K-semistable. There exists a torus invariant   $\Q$-divisor $D^* \in |-K_X-\D|_{\Q}$ such that
\begin{itemize}
    \item[(1)] $(X,\D+(1-\d(X,\D))D^*)$ is a log Fano pair and
      \item[(2)] $\d(X,\D+(1-\d(X,\D))D^*) =1$.  
\end{itemize}
\end{proposition}

\begin{proof}
Let $\overline{u}$ denote the barycenter of $P_{-K_{X}-\D}$ and  $c$ the largest real number such that $-c\overline{u} \in P_{-K_{X}-\D}$. Recall that $\delta(X,\D)= c/(1+c)$ by Proposition \ref{p:TlFd}.2. 
Set
$$D^* : = D_{-c\overline{u}} = \sum_{i=1}^d \left((1-b_i) +
\langle-c \overline{u}, v_i \rangle \right) D_i\in |-K_X - \D |_{\Q}.$$

We first show statement (1). For $i=1,\ldots,d$, we compute
\begin{align*} 
A_{X,\D+(1-\d(X,\D))D^*}(v_i) 
&= 
1-b_i - (1-c/(c+1))( 1 -b_i +
\langle - c\overline{u}, v_i \rangle )\\
&= (c/(c+1)) \left( 1-b_i + \langle \overline{u},v_i \rangle \right).
\end{align*}
Since $\overline{u}$ is in the interior of $P_{-K_{X}-\D}$, $\langle \overline{u},v_i \rangle > -1+b_i$. Hence, the above log discrepancies are $>0$ and the pair is klt. Since $-(K_X+\D+(1-\d(X,\D))D^*) \sim_{\Q} -\d(X,\D) (K_X+\D)$ is ample, $(X,\D+(1-\d(X,\D))D^*)$ is log Fano.

To prove (2), we compute
\begin{align*}
S(-(K_X+\D+(1-\d(X,\D)D^*);v_i) & = \d(X,\D) S(-(K_X-\D);v_i) \\
                            & = (c/(c+1)) (1-b_i + \langle \overline{u},v_i \rangle ). 
\end{align*}
 Proposition \ref{p:Tdelta} and our previous computations imply
$\d(X,\D+(1-\d(X,\D))D^*) = 1$. 
\end{proof}

\begin{theorem}\label{p:Tgenlog}
Let $(X,\D)$ be a toric log Fano pair. If $m \in \Z_{>0}$ is sufficiently divisible and $D= m^{-1}H$, where $H \in |-m(K_X+\D)|$ is very general, then 
$$(X,\D+(1-\beta)D)  \text{ is uniformly K-stable for } \beta \in (0,\d(X,\D)).$$
Moreover, $(X,\D+(1-\d(X,\D))D)$ is K-semistable.
\end{theorem}

\begin{proof}
If $\d(X,\D) =1$, the statement follows from Proposition \ref{p:SWformula}. From now on,  assume $\d(X,\D)<1$. 

Let $D^*\in |-(K_X+\D)|_{\Q}$ denote a torus invariant $\Q$-divisor satisfying Proposition \ref{l:T*delta}. 
Fix a  integer $m\geq 2$ so that $mD^*$ is a Cartier divisor and $|-m(K_X+\D)|$ is base point free. 

We claim that for a very general $H\in |-m(K_X+\D)|$ (i) $(X,\D+H)$ is lc and (ii) $\d(X,\D+(1-\d(X,\D))m^{-1} H)\geq 1$. Indeed, since $|-m(K_X+\D)|$ is base point free, (i) follows from \cite[Lemma 5.17]{KM98}. Since
\[
\bigcap_{q \in \Z_{>0}}  \Big\{ H \in |-m(K_X+\D)| \,\, \Big\vert \,\, \d(X,\D+(1-\d(X,\D))m^{-1}H)>1-1/q \Big\}
\]
is a countable intersection of open sets (by Theorem \ref{t:deltalsc}) and  each contains $mD^*$, (ii) holds.  

Now, consider a very general element $H \in |-m(K_X+\D)|$ satisfying  (i) and (ii). Set $D:= m^{-1} H$. 
We claim that $\d(X,\D+(1-\beta)D)$ is decreasing in $\beta$ on $(0,1]$. 
Assuming the claim, (ii) implies  $\d(X,\D+(1-\beta)D)> 1$ and, hence, $(X,\D+(1-\beta)D)$ is uniformly K-stable for $\beta \in (0,\d(X,\D))$. 

To prove the above claim, it suffices to show that for each $v\in \Val_X^*$ with $A_{X,\D}(v)< +\infty$, \begin{equation}\label{e:ASbeta}
\frac{A_{X,\D+(1-\beta)D}(v)}{S(- (K_{X}+\D+(1-\beta)D);v)}\end{equation}
is a differentiable  function in $\beta$ with negative derivative bounded away from 0  by $-((m-1)/m)\d(X,\D)/\beta^2$. 
To see this, we compute 
\begin{align*}
\frac{d}{d\beta} \left(\frac{A_{X,\D+(1-\beta)D}(v)}{S(- (K_{X}+\D+(1-\beta)D);v)} \right)
& = \frac{d}{d\beta} \left(\frac{ A_{X,\D+D}(v)+ \beta v(D)}{\beta S(-K_X-\D;v)}\right) \\
& = -\frac{ A_{X,\D+D}(v)}{\beta ^2S(-K_X-\D;v)}
\intertext{Since 
    $(X,\D+mD)$ is lc by (i), $A_{X,\D+mD}(v)= A_{X,\D}(v)-mv(D)\geq0$. Thus, $A_{X,\D+D}(v) = A_{X,\D}(v) - v(D) \geq (m-1)m^{-1} A_{X,\D}(v)$, and
    }
& \leq -\left( \frac{m-1}{m} \right) \frac{ A_{X,\D}(v)}{\beta ^2S(-K_X-\D;v)}\\
& \leq -\left( \frac{m-1}{m} \right) \frac{ \d(X,\D)}{\beta ^2}.
\end{align*}
This completes the proof.
\end{proof}

\end{document}